\documentclass[bj,numbers,noshowframe]{imsart}

\RequirePackage{amsthm,amsmath,amsfonts,amssymb,mathtools}
\startlocaldefs
\theoremstyle{plain}
\newtheorem{theorem}{Theorem}
\newtheorem{lemma}{Lemma}
\newtheorem{corollary}{Corollary}
\newtheorem{proposition}{Proposition}
\DeclareMathOperator{\dCov}{dCov}
\DeclareMathOperator{\dCor}{dCor}
\endlocaldefs

\begin{document}

\begin{frontmatter}
\title{Bergsma--Dassios Sign Covariance Characterises Independence for
Arbitrary Real-Valued Bivariate Laws}
\runtitle{Sign Covariance Characterises Independence}

\begin{aug}
\author[A]{\inits{S.}\fnms{Stefan}~\snm{Gr\"unewald}
\ead[label=e1]{stefan@picb.ac.cn}}
\author[B]{\inits{L.}\fnms{Libo}~\snm{Huang}
\ead[label=e2]{917621923@qq.com}}
\runauthor{Gr\"unewald and Huang}
\address[A]{Shanghai Institute of Nutrition and Health, Chinese Academy of
Sciences, Shanghai, China\printead[presep={,\ }]{e1}}
\address[B]{Innovation School, Jiaxiang Education Group, Chengdu, Sichuan,
China\printead[presep={,\ }]{e2}}
\end{aug}

\begin{abstract}
Bergsma--Dassios sign covariance $\tau^*$ is a rank-based population
measure of dependence.  Building on zero-characterisation results under
specific regularity regimes, we prove that $\tau^*(X,Y)=0$ characterises
independence for every real-valued bivariate distribution, including mixed
and singular laws.  For the unnormalised four-sample convention for $\tau^*$
used here and the unscaled Blum--Kiefer--Rosenblatt functional $\mathcal B$,
the proof gives the
quantitative inequality $\tau^*\ge 2\mathcal B$.  This is a population
identification result; no new sample-level limit theorem is claimed.  The
argument first encodes finite ordered distributions with rational cell
probabilities by labelled path trees and applies a nonnegative
sum-of-squares representation for a quartet covariance.  Rational
approximation and nested quantisation then remove all support and regularity
restrictions.  On finite uniformly weighted label sets, the tree framework
also relates an edge-weighted quartet quantity to empirical distance
covariance squared.  As a separate combinatorial consequence, it yields the
asymptotic $2/3$ upper bound for the quartet distance between binary phylogenetic
trees.
\end{abstract}


\end{frontmatter}

\section{Introduction}

Testing independence without imposing a parametric model is a central
problem of multivariate statistics.  Rank- and sign-based procedures are
particularly attractive because they are invariant under strictly increasing
transformations and can detect forms of dependence missed by
linear correlation.  Bergsma and Dassios~\cite{BergsmaDassios2014}
introduced the sign covariance $\tau^*$ as an extension of Kendall's $\tau$.
They proved that $\tau^*\ge0$, with equality if and only if the variables
are independent, when the joint law is discrete, jointly absolutely
continuous, or a mixture of these two types, and conjectured the same
conclusion for arbitrary bivariate laws.  These developments belong to a
longer line of pattern-based independence statistics.  Hoeffding's rank test
is based on patterns of five observations~\cite{Hoeffding1948}, whereas
Yanagimoto~\cite{Yanagimoto1970}, under the continuity assumptions of his
paper, presented a four-observation measure of association.  Drton, Han, and
Shi~\cite{DrtonHanShi2020} later identified that measure, for absolutely
continuous pairs, as proportional to the Bergsma--Dassios statistic; this
historical connection underlies later usage of the name
Bergsma--Dassios--Yanagimoto $\tau^*$.  They also established both
nonnegativity and the zero characterisation for random vectors with
continuous margins, allowing joint laws that need not be absolutely
continuous.  In the no-ties sample setting, and with their normalisation
conventions, they derived an exact identity linking the associated
$U$-statistics for Hoeffding's $D$, the Blum--Kiefer--Rosenblatt statistic
$R$, and $\tau^*$.  This sample-level identity is distinct from the population
inequality proved below.  The empirical statistic is a bounded
$U$-statistic: Nandy, Weihs, and Drton~\cite{NandyWeihsDrton2016} developed
its large-sample theory, while Weihs, Drton, and
Meinshausen~\cite{WeihsDrtonMeinshausen2018} placed it in the broader
framework of symmetric rank covariances.  More recently, Baringhaus and
Gr\"ubel~\cite{BaringhausGruebel2026} placed these and related procedures in
a general framework of pattern-based independence tests and studied their
limiting null distributions and local asymptotic relative efficiencies.
These historical and sample-level results are complementary to the
population identification result developed below.

\noindent\textbf{Main statistical result.}
Extending the cases covered by the preceding results,
Corollary~\ref{cor:tau-star} proves that $\tau^*(X,Y)>0$ whenever the
real-valued random variables $X$ and $Y$ are dependent, without excluding
atoms, mixtures, or singular components.  More precisely, for the
unnormalised four-sample convention for $\tau^*$ defined in
Subsection~\ref{subsec:tau-star} and the unscaled
Blum--Kiefer--Rosenblatt functional $\mathcal B$~\cite{BlumKieferRosenblatt1961}
defined in
\eqref{eq:BKR-discrepancy}, the proof establishes
\eqref{eq:tau-star-BKR-lower-bound}, namely $\tau^*(X,Y)\ge2\mathcal B(X,Y)$.
Dependence makes $\mathcal B$ positive and hence forces $\tau^*>0$.  In the
reverse direction, independence factors the expectation defining $\tau^*$,
and exchangeability makes each factor zero.  Thus $\tau^*=0$ characterises
independence for every real bivariate law.  No density, continuity, or moment
assumption is used.  The result concerns population identification; no new
sample-level limit theorem is proved here.

The proof begins with a finite combinatorial construction.  A finite ordered bivariate
distribution with rational cell probabilities is represented by two labelled
path trees.  On general
labelled trees we define quartet covariance, a signed comparison of the
four-sample partitions displayed by the two trees, and prove that it is
a sum of four families of nonnegative squares.  The coefficient lower
bound extracted from that representation yields the finite-support
inequality for $\tau^*$.  Rational approximation handles arbitrary
finite tables, while nested quantisation and dominated convergence pass
to arbitrary real laws.  The finite combinatorial identity and the
measure-theoretic limiting argument are established separately.
The extension to singular distributions uses bounded weak-order kernels,
right-closed nested quantisers, continuity from above, and dominated
convergence.

The tree representation also connects two statistical and combinatorial
notions that are usually studied separately.  On a finite uniformly weighted
label set, Theorem~\ref{thm:tree-distance-covariance} identifies the
edge-weighted quartet quantity with $3/2$ times empirical distance covariance
squared for the induced tree pseudometrics.  For binary phylogenetic trees on
$n$ taxa,
the sum-of-squares theorem also implies the asymptotic $2/3$ upper bound for
quartet distance conjectured by Bandelt and Dress~\cite{BandeltDress1986}.
Earlier work obtained the upper bound
$0.69\binom{n}{4}+o(n^4)$~\cite{AlonNavesSudakov2016}.  A 2026 preprint by
Pachter~\cite{Pachter2026} also proves the $2/3$ asymptotic result, by a
different argument based on common-root planarisation and a five-leaf
identity.  In the present paper, the quartet-distance bound is a separate
combinatorial corollary.  The input needed for the statistical result is the
coefficient estimate extracted from our sum-of-squares representation, not
the $2/3$ conclusion itself.

Section~\ref{sec:tree-framework} introduces labelled trees, partial quartets,
and quartet covariance.
Section~\ref{sec:sos} proves the inclusion--exclusion and sum-of-squares identities and,
most importantly for the statistical result, the coefficient estimate used
later.  The same section also records the quartet-distance corollary.
Section~\ref{sec:statistical-consequences} gives two related tree-dependence constructions and proves the
result for arbitrary real bivariate laws in
Subsection~\ref{subsec:tau-star}.

The proof of Corollary~\ref{cor:tau-star} uses the sum-of-squares identity and
its coefficient lemma only through their specialisation to path trees.  The
normalised tree indices, metric-distance identity, and phylogenetic extremal
bound are separate consequences and are not used in that proof.

\section{Tree and quartet covariance framework}
\label{sec:tree-framework}

We follow Semple and Steel~\cite{SempleSteel2003} for standard terminology and
notation concerning phylogenetic $X$-trees, splits, and displayed quartets.

For a vertex $v$ of a graph $G$, let $N(v)$ denote the set of neighbours
of $v$ and let $N_+(v)=N(v)\cup\{v\}$.  For $w\in N(v)$, define
$N_v(w):=N(w)\setminus\{v\}$ and $N_{v+}(w):=N_v(w)\cup\{w\}$.
Thus, for every ordered adjacent pair $(t,u)$ in $G$,
$N_{t+}(u)=N_t(u)\cup\{u\}$.  For such an ordered adjacent pair, define the
following two sets of ordered vertex pairs:
\[
 I_{t,u}:=\{(u',u'')\in N_{t+}(u)\times N_+(u):
 u'=u\text{ or }u''\notin\{t,u'\}\}
\]
and
\[
 I'_{t,u}:=\{(u',u'')\in N_{t+}(u)\times N_+(u):
 u'=u\text{ or }u''\ne u'\}.
\]

For a finite nonempty set $X$, an $X$-tree $\mathcal T=(T,\phi)$ is a finite
(graph-theoretic) tree $T=(V,E)$ together with a labelling map
$\phi:X\to V$, where $\phi^{-1}(v)\ne\varnothing$ for every vertex $v$
of degree at most $2$.  An $X$-tree is called \emph{trivial} if its underlying
tree has no edge, and \emph{non-trivial} otherwise.  A trivial $X$-tree
displays no partial quartet.  An $X$-tree is called \emph{phylogenetic} if
$\phi$ is a bijection from $X$ onto the set of graph-theoretic leaves of
$T$ (vertices of degree $1$); thus a phylogenetic
$X$-tree cannot have vertices of degree $2$.  A phylogenetic tree is
called \emph{binary} if all nonleaf vertices have degree $3$.

A \emph{partial split} of $X$ is an unordered pair of non-overlapping
non-empty subsets of $X$.  A partial split $\{X_1,X_2\}$, with
$X_1=\{x_1^1,\ldots,x_1^i\}$ and
$X_2=\{x_2^1,\ldots,x_2^j\}$, is denoted by $X_1\mid X_2$ or
$x_1^1\cdots x_1^i\mid x_2^1\cdots x_2^j$.  A partial split
$X_1\mid X_2$ is called a \emph{full split}, or simply a \emph{split},
if $X_1\cup X_2=X$; it is called a \emph{partial quartet} if
$|X_1|,|X_2|\le2$, and a \emph{(full) quartet} if
$|X_1|=|X_2|=2$.  Partial quartets $x_1\mid x_2$ and
$x_1\mid x_2x_3$ will sometimes be denoted by
$x_1x_1\mid x_2x_2$ and $x_1x_1\mid x_2x_3$, respectively.

For two partial splits $X_1\mid X_2$ and $X_3\mid X_4$ of $X$, we say
that $X_1\mid X_2$ \emph{displays} $X_3\mid X_4$ if either
$X_3\subseteq X_1$ and $X_4\subseteq X_2$, or $X_3\subseteq X_2$ and
$X_4\subseteq X_1$.  For an $X$-tree $\mathcal T=(T,\phi)$ and an edge
$uv$ of $T$, let $V^u(v)$ be the set of vertices in the same component
of $T-uv$ as $v$, and let
\[
 S_{uv}:=\bigcup_{v'\in V^u(v)}\phi^{-1}(v')
 \mathbin{\Big|}
 \bigcup_{u'\in V^v(u)}\phi^{-1}(u')
\]
be the corresponding split of $X$.  For $w\in N_u(v)$, let
\[
 S_{u,w}:=\bigcup_{u'\in V^v(u)}\phi^{-1}(u')
 \mathbin{\Big|}
\bigcup_{w'\in V^v(w)}\phi^{-1}(w')
\]
be the corresponding partial split of $X$.
Here $u-v-w$ is the unique length-$2$ path: the middle vertex $v$, its
labels, and all branches at $v$ other than the two endpoint branches are
omitted from the two displayed sides.

When $H$ is an edge $uv$, write $S_H:=S_{uv}$; when $H$ is the
length-$2$ path $u-v-w$, write $S_H:=S_{u,w}$.  Thus $S_e$ for an edge
$e$ always means its edge split, whereas $S_H$ for a length-$2$ path means
its endpoint partial split.  In later formulas whose surrounding indices
specify that $u,w$ are the distance-$2$ endpoints of $u-v-w$, the shorthand
$S_{uw}$ also denotes this endpoint partial split $S_{u,w}$; the
distance-$2$ condition distinguishes this usage from the edge-split notation.

Every component of a finite tree obtained by deleting an edge contains a
leaf of the original tree.  Such a leaf has a nonempty label fibre by the
definition of an $X$-tree.  Consequently both sides of every edge split are
nonempty.  Likewise, the union of the label fibres in the component reached
through any neighbouring branch is nonempty; the label fibre at the central
vertex itself may be empty.

A partial split $X_1\mid X_2$ is \emph{displayed} by $\mathcal T$ if
there is an edge $uv$ of $T$ such that $S_{uv}$ displays
$X_1\mid X_2$.  For $i\in\{2,3,4\}$, let $Q_i(\mathcal T)$ denote the
set of all displayed partial quartets containing precisely $i$ taxa.
Associate with a partial quartet $q=A\mid B$ a four-element multiset
$M(q)$ by listing each member of a two-element side once and, for every
singleton side, listing its unique taxon twice.  A quadruple
$(x_1,x_2,x_3,x_4)$ \emph{induces} $q$ if its multiset of entries is
$M(q)$.  Here the support of a quadruple is the set of its distinct entries.
Thus a support-$2$ quadruple induces a partial quartet precisely
for multiplicities $2+2$, and a support-$3$ quadruple does so precisely for
multiplicities $2+1+1$; in either case the induced object is unique.  A
support-$4$ quadruple induces each of the three possible full quartet
partitions on that support, while multiplicities $3+1$ or $4$ induce none.

Let $\mathcal T_1=(T_1,\phi_1)$ and
$\mathcal T_2=(T_2,\phi_2)$ be two $X$-trees with
$T_1=(V_1,E_1)$ and $T_2=(V_2,E_2)$.  For $i\in\{1,2\}$, let $D_i$
be the set of ordered pairs of adjacent vertices in $T_i$, let $E_i^2$
be the set of unordered pairs of vertices at distance $2$, and let
\[
 D_i^2:=\{(u,v,w)\in V_i^3:uv,vw\in E_i\text{ and }u\ne w\}
\]
be the set of ordered paths of length $2$ in $T_i$.  Fix an arbitrary total
order $<_i$ on $V_i$ and put
\[
 P_i:=\{(u,v,w)\in D_i^2:u<_i w\}.
\]
Thus $P_i$ contains exactly one representative of each reversal orbit
$\{(u,v,w),(w,v,u)\}$.  Reversing the endpoints leaves the endpoint partial
split unchanged.  A summation condition $uvw\in E_i^2$ is shorthand for
$\{u,w\}\in E_i^2$, with $v$ the unique middle vertex; it displays the middle
vertex without changing $E_i^2$ from a set of unordered endpoint pairs.

Let $(t,u)\in D_1$ and $(v,w)\in D_2$.  Throughout this paragraph,
neighbour and component notation involving $t,u$ is taken in $T_1$, while
that involving $v,w$ is taken in $T_2$.  Define
$R_{t,u}(u)=\phi_1^{-1}(u)$.  For $u'\in N(u)$, let $R_{t,u}(u')$
denote the union of the label sets $\phi_1^{-1}(u'')$ over all
$u''\in V_1^u(u')$.  Correspondingly, define
$C^{v,w}(w)=\phi_2^{-1}(w)$, and for $w'\in N(w)$ let
$C^{v,w}(w')$ be the union of $\phi_2^{-1}(w'')$ over all
$w''\in V_2^w(w')$.  For $u'\in N_+(u)$ and $w'\in N_+(w)$, put
\[
 A_{t,u}^{v,w}(u',w'):=R_{t,u}(u')\cap C^{v,w}(w').
\]
We call the sets $R_{t,u}(u')$ the row blocks, the sets $C^{v,w}(w')$
the column blocks, and their intersections $A_{t,u}^{v,w}(u',w')$ the
cells of the resulting row--column table.
For the block sets $R_{t,u}(u')$, $C^{v,w}(w')$, and
$A_{t,u}^{v,w}(u',w')$ just defined, and for their later indexed variants,
the corresponding lowercase letter denotes normalized cardinality, namely
cardinality divided by $|X|$; for example,
$r_{t,u}(u')=\frac{|R_{t,u}(u')|}{|X|}$.
For every $u'\in N(u)$ and $w'\in N(w)$, the sets $R_{t,u}(u')$ and
$C^{v,w}(w')$ are the nonempty branch-component label sets described above.
Consequently, $r_{t,u}(u')>0$ for $u'\in N(u)$, and
$c^{v,w}(w')>0$ for $w'\in N(w)$.
These are precisely the neighbouring-branch masses used as row sums or
denominators in the coefficient calculations below.  No positivity is
asserted for the possibly empty central fibres $R_{t,u}(u)$ and
$C^{v,w}(w)$.

The coefficient family denoted by lowercase $q$ below is unrelated to the
temporary letter $q$ used above for a generic partial quartet.  For
$(t,u)\in D_1$, $(v,w)\in D_2$, $u'\in N_{t+}(u)$,
$u''\in N_+(u)$, $w'\in N_{v+}(w)$, and $w''\in N_+(w)$, define
\[
q_{t,u}^{v,w}(u',u'',w',w'')
:=\frac12c^{v,w}(w'')\left(
 a_{t,u}^{v,w}(u',w')\frac{a_{t,u}^{v,w}(u'',v)}{c^{v,w}(v)}
+
 a_{t,u}^{v,w}(u'',w')\frac{a_{t,u}^{v,w}(u',v)}{c^{v,w}(v)}
\right).
\]
If $u'\in N_t(u)$, then
\[
q_{t,u}^{v,w}(u',t,w',w'')
=q_{u',u}^{v,w}(t,u',w',w''),
\]
and, retaining the two endpoints $t,u'$ of the path $t-u-u'$, we define
\[
b_{tu'}^{v,w}(w',w''):=q_{t,u}^{v,w}(u',t,w',w'').
\]
Correspondingly, $w'\in N_v(w)$ implies
\[
q_{t,u}^{v,w}(u',u'',w',v)
=q_{t,u}^{w',w}(u',u'',v,w'),
\]
and we define
\[
s_{t,u}^{vw'}(u',u''):=q_{t,u}^{v,w}(u',u'',w',v).
\]
If $u'\in N_t(u)$ and $w'\in N_v(w)$, define
$p_{tu'}^{vw'}:=q_{t,u}^{v,w}(u',t,w',v)$.
Thus each displayed $b$-, $s$-, or $p$-coefficient is obtained from the
displayed $q$-coefficient by imposing, respectively, $u''=t$, $w''=v$, or
both, at the stated indices and then using the displayed orientation
identity.  These are indexed coefficient identities, not identities between
complete coefficient families.  The lowercase coefficient $p$ is unrelated
to the path representative sets $P_i$.
Then
\begin{align*}
p_{tu'}^{vw'}
={}&q_{u',u}^{v,w}(t,u',w',v)
=q_{t,u}^{w',w}(u',t,v,w')
=q_{u',u}^{w',w}(t,u',v,w')\\
={}&b_{tu'}^{v,w}(w',v)
=b_{tu'}^{w',w}(v,w')
=s_{t,u}^{vw'}(u',t)
=s_{u',u}^{vw'}(t,u').
\end{align*}
The preceding identities show that $b$ and $s$ are unchanged by reversal of
the associated path endpoints, and that $p$ is unchanged by either reversal.
The corresponding squared linear forms in
Theorem~\ref{thm:sum-of-squares} are also unchanged.  Hence the sums over
$P_i$ do not depend on the auxiliary total order $<_i$.

Throughout this finite-label framework, drawing taxa with replacement means
independent uniform draws from $X$; consequently every ordered quadruple has
probability $|X|^{-4}$.  In the statistical application, a rational
probability table is represented by replicating labels so that these uniform
label frequencies equal the cell probabilities.

We define the \emph{quartet covariance} $Q(\mathcal T_1,\mathcal T_2)$
to be the probability of drawing four taxa with replacement that induce
a partial quartet displayed by both trees, minus half the probability
of drawing four taxa that induce one quartet displayed by $\mathcal T_1$
and a different quartet displayed by $\mathcal T_2$.

For two partial $X$-splits $S_1=A\mid B$ and $S_2=C\mid D$, define
$Q(S_1,S_2)$ analogously.  Then
\[
 Q(S_1,S_2)=\frac6{|X|^4}
 \left(\det\!\begin{pmatrix}
 |A\cap C|&|A\cap D|\\
 |B\cap C|&|B\cap D|
 \end{pmatrix}
 \right)^2.
\]
We call $S_1,S_2$ \emph{independent} if $Q(S_1,S_2)=0$, and call
$\mathcal T_1,\mathcal T_2$ independent if $S_{tu},S_{vw}$ are
independent for every $tu\in E_1$ and $vw\in E_2$.

For a four-element subset $Y\subseteq X$, an $X$-tree displays at most one
full quartet with support $Y$.  Define
\[
d_Q(\mathcal T_1,\mathcal T_2)
:=\#\{Y\subseteq X:|Y|=4,\ \mathcal T_1\text{ and }\mathcal T_2
\text{ display distinct full quartets on }Y\}.
\]
For binary phylogenetic trees every four-set supports exactly one displayed
quartet, so this definition agrees with the usual quartet
distance~\cite{SteelPenny1993}.

Finally, let
\[
X_{t,u}^{v,w}:=
a_{t,u}^{v,w}(t,v)a_{u,t}^{w,v}(u,w)
-a_{t,u}^{w,v}(t,w)a_{u,t}^{v,w}(u,v).
\]
Then
\[
X_{t,u}^{v,w}=-X_{t,u}^{w,v}=-X_{u,t}^{v,w}=X_{u,t}^{w,v},
\qquad
Q(S_{tu},S_{vw})=6\bigl(X_{t,u}^{v,w}\bigr)^2.
\]
By a \emph{pure-square term} we mean a term proportional to the square of a
single $X_{t,u}^{v,w}$ after one of the squared linear forms in
Theorem~\ref{thm:sum-of-squares} is expanded.  An individual pure-square
occurrence associated with one endpoint orientation need not be invariant.
Whenever a representative in $P_i$ is used, the formulas below include the
pure-square occurrences associated with both endpoint orientations.

\section{Sum-of-squares identity and coefficient bound}
\label{sec:sos}

Using the inclusion--exclusion principle, we can compute the quartet
covariance between two $X$-trees in terms of the splits and partial
splits corresponding to edges and paths of length $2$.

Theorem~\ref{thm:sum-of-squares} gives the representation used below.  For the
statistical application, we need Lemma~\ref{lem:coefficient-lower-bound},
which is obtained by retaining the first nonnegative sum and gives the local
quantitative bound later applied to path trees in the proof of
Corollary~\ref{cor:tau-star}.  The nonnegativity of the full identity also
gives the quartet-distance bound in Corollary~\ref{cor:two-thirds}.

\begin{lemma}\label{lem:display-path}
Let $q=A\mid B$ be a partial quartet displayed by an $X$-tree
$\mathcal T$.  The edges $e$ for which $S_e$ displays $q$ are the edges of
one nonempty path.  If that path has $k$ edges, then the length-$2$ paths
whose endpoint partial split displays $q$ are exactly its $k-1$ consecutive
edge pairs.
\end{lemma}

\begin{proof}
Let $H_A$ and $H_B$ be the minimal connected subtrees spanning the labelled
vertices of the taxa in $A$ and $B$, respectively.  Since some edge split
displays $A\mid B$, these two subtrees are disjoint.  An edge separates all
of $A$ from all of $B$ exactly when it lies on the unique path joining
$H_A$ to $H_B$.  Two such separating edges are adjacent exactly when their
outer endpoint branches define one of the partial splits $S_{u,w}$ above.
The last assertion is therefore the elementary count of consecutive edge
pairs in a path.
\end{proof}

\begin{theorem}\label{thm:inclusion-exclusion}
For every two $X$-trees $\mathcal T_1$ and $\mathcal T_2$, we have
\begin{align*}
Q(\mathcal T_1,\mathcal T_2)
={}&\sum_{\substack{tu\in E_1\\vw\in E_2}}Q(S_{tu},S_{vw})
-\sum_{\substack{tu\in E_1\\v_1wv_2\in E_2^2}}Q(S_{tu},S_{v_1v_2})\\
&-\sum_{\substack{t_1ut_2\in E_1^2\\vw\in E_2}}Q(S_{t_1t_2},S_{vw})
+\sum_{\substack{t_1ut_2\in E_1^2\\v_1wv_2\in E_2^2}}
Q(S_{t_1t_2},S_{v_1v_2}).
\end{align*}
\end{theorem}

\begin{proof}
By Lemma~\ref{lem:display-path}, for every displayed partial quartet $q_i$
the number of displaying length-$2$ paths is one less than the number of
displaying edges.

Let $(x_1,x_2,x_3,x_4)$ induce a partial quartet $q_i$ displayed by
$\mathcal T_i$, for $i\in\{1,2\}$.  Let $H_i$ be an edge or a path of
length $2$ of $T_i$ such that $S_{H_i}$ displays $q_i$.  The contribution
of the quadruple to both $Q(\mathcal T_1,\mathcal T_2)$ and
$Q(S_{H_1},S_{H_2})$ is $|X|^{-4}$ if $q_1=q_2$, and
$-(2|X|^4)^{-1}$ otherwise.  Denote this contribution by
$\gamma(x_1,x_2,x_3,x_4)$, and let $k_i$ be the number of edges $u_iv_i$
of $T_i$ for which $S_{u_iv_i}$ displays $q_i$.  The contribution to the
right-hand side is
\[
\gamma(x_1,x_2,x_3,x_4)
\bigl(k_1k_2-k_1(k_2-1)-(k_1-1)k_2+(k_1-1)(k_2-1)\bigr)
=\gamma(x_1,x_2,x_3,x_4).
\]
Summing over all ordered quadruples proves the formula.
\end{proof}

We will show that the quartet covariance is a sum of squares.

\begin{theorem}\label{thm:sum-of-squares}
For every two $X$-trees $\mathcal T_1$ and $\mathcal T_2$,
\begin{align*}
\frac16Q(\mathcal T_1,\mathcal T_2)
={}&\sum_{\substack{(t,u)\in D_1\\(v,w)\in D_2}}
\sum_{\substack{(u',u'')\in I_{t,u}\\(w',w'')\in I_{v,w}}}
q_{t,u}^{v,w}(u',u'',w',w'')\bigl(X_{t,u}^{v,w}\bigr)^2\\
&+\sum_{\substack{(t,u)\in D_1\\(v_1,w,v_2)\in P_2}}
\sum_{(u',u'')\in I_{t,u}}
s_{t,u}^{v_1v_2}(u',u'')
\bigl(X_{t,u}^{v_1,w}+X_{t,u}^{v_2,w}\bigr)^2\\
&+\sum_{\substack{(t_1,u,t_2)\in P_1\\(v,w)\in D_2}}
\sum_{(w',w'')\in I_{v,w}}
b_{t_1t_2}^{v,w}(w',w'')
\bigl(X_{t_1,u}^{v,w}+X_{t_2,u}^{v,w}\bigr)^2\\
&+\sum_{\substack{(t_1,u,t_2)\in P_1\\(v_1,w,v_2)\in P_2}}
p_{t_1t_2}^{v_1v_2}
\bigl(X_{t_1,u}^{v_1,w}+X_{t_1,u}^{v_2,w}
+X_{t_2,u}^{v_1,w}+X_{t_2,u}^{v_2,w}\bigr)^2.
\end{align*}
\end{theorem}

To prove Theorem~\ref{thm:sum-of-squares}, we express the terms involving
one or two partial splits in terms of the $X$-variables.  We use the
following elementary linear-algebra identity.

\begin{lemma}\label{lem:matrix-identity}
Let
\[
A=\begin{pmatrix}a_1^1&a_1^2&a_1^3\\a_2^1&a_2^2&a_2^3\end{pmatrix}
\]
be a real $2\times3$ matrix.  For $i,j\in\{1,2,3\}$, put
\[
A^{i,j}:=\begin{pmatrix}a_1^i&a_1^j\\a_2^i&a_2^j\end{pmatrix},
\quad
A^{1,2+3}:=\begin{pmatrix}a_1^1&a_1^2+a_1^3\\
a_2^1&a_2^2+a_2^3\end{pmatrix},
\]
and
\[
A^{3,1+2}:=\begin{pmatrix}a_1^3&a_1^2+a_1^1\\
a_2^3&a_2^2+a_2^1\end{pmatrix}.
\]
Let $r_1=a_1^1+a_1^2+a_1^3$ and
$r_2=a_2^1+a_2^2+a_2^3$, and assume $r_1,r_2\ne0$.  For every
$\alpha\in\mathbb R$,
\[
\det(A^{1,3})=
\left(\alpha\frac{a_1^3}{r_1}+(1-\alpha)\frac{a_2^3}{r_2}\right)
\det(A^{1,2+3})
-\left(\alpha\frac{a_1^1}{r_1}+(1-\alpha)\frac{a_2^1}{r_2}\right)
\det(A^{3,1+2}).
\]
\end{lemma}

\begin{proof}
Append a copy of the first row of $A$ above $A$ to form a $3\times3$
matrix $A_+$.  Its first two rows coincide, so $\det(A_+)=0$.  Expanding
along the first row gives
\[
0=a_1^1\det(A^{2,3})-a_1^2\det(A^{1,3})+a_1^3\det(A^{1,2}),
\]
and hence
\[
a_1^2\det(A^{1,3})=a_1^1\det(A^{2,3})+a_1^3\det(A^{1,2}).
\]
Adding $(a_1^1+a_1^3)\det(A^{1,3})$ and dividing by $r_1$ proves the
formula for $\alpha=1$, because
$\det(A^{1,2+3})=\det(A^{1,2})+\det(A^{1,3})$ and
$\det(A^{3,1+2})=-\det(A^{1,3})-\det(A^{2,3})$.  Appending a copy of the second
row instead proves the case $\alpha=0$.  The general identity is the
linear combination of these two cases with coefficients $\alpha$ and
$1-\alpha$.
\end{proof}

\begin{proof}[Proof of Theorem~\ref{thm:sum-of-squares}]
We express the four contributions in Theorem~\ref{thm:inclusion-exclusion}
as quadratic functions of at most four variables $X_{t,u}^{v,w}$.
For any $h\in\{1,2\}$, write $\bar h:=3-h$ for the complementary index.
For edges $tu$ of $T_1$ and $vw$ of $T_2$,
\[
\frac16Q(S_{tu},S_{vw})=(X_{t,u}^{v,w})^2,
\]
and this remains true after interchanging $t,u$ and/or $v,w$.  Directly
from the definition of the coefficient family $q$ and the row and column
partitions,
\begin{align*}
1={}&
\sum_{\substack{u'\in N_{t+}(u),\ u''\in N_+(u)\\
w'\in N_{v+}(w),\ w''\in N_+(w)}}
q_{t,u}^{v,w}(u',u'',w',w'')
+\sum_{\substack{u'\in N_{t+}(u),\ u''\in N_+(u)\\
v'\in N_{w+}(v),\ v''\in N_+(v)}}
q_{t,u}^{w,v}(u',u'',v',v'')\\
&+\sum_{\substack{t'\in N_{u+}(t),\ t''\in N_+(t)\\
w'\in N_{v+}(w),\ w''\in N_+(w)}}
q_{u,t}^{v,w}(t',t'',w',w'')
+\sum_{\substack{t'\in N_{u+}(t),\ t''\in N_+(t)\\
v'\in N_{w+}(v),\ v''\in N_+(v)}}
q_{u,t}^{w,v}(t',t'',v',v'').
\end{align*}
Associating each coefficient with the $X$-variable having the same outer
directed-edge indices $(t,u)$ and $(v,w)$ gives
\begin{align}
\frac16\sum_{\substack{tu\in E_1\\vw\in E_2}}Q(S_{tu},S_{vw})
={}&\sum_{\substack{(t,u)\in D_1\\(v,w)\in D_2}}
\sum_{\substack{u'\in N_{t+}(u),\ u''\in N_+(u)\\
w'\in N_{v+}(w),\ w''\in N_+(w)}}
q_{t,u}^{v,w}(u',u'',w',w'')
(X_{t,u}^{v,w})^2.                                      \label{eq:edge-edge}
\end{align}

For an edge $tu$ of $T_1$ and a path $v_1wv_2$ of length $2$ in $T_2$,
\[
\frac16Q(S_{tu},S_{v_1v_2})=
\left|
\begin{matrix}
a_{t,u}^{v_1,w}(t,v_2)&a_{t,u}^{v_2,w}(t,v_1)\\
a_{u,t}^{v_1,w}(u,v_2)&a_{u,t}^{v_2,w}(u,v_1)
\end{matrix}
\right|^2.
\]
Apply Lemma~\ref{lem:matrix-identity}, with $\alpha=r_{t,u}(t)$, to
\begingroup
\small
\setlength{\arraycolsep}{12pt}
\renewcommand{\arraystretch}{1.12}
\[
\begin{pmatrix}
a_{t,u}^{v_1,w}(t,v_2)&
r_{t,u}(t)-a_{t,u}^{v_1,w}(t,v_2)-a_{t,u}^{v_2,w}(t,v_1)&
a_{t,u}^{v_2,w}(t,v_1)\\
a_{u,t}^{v_1,w}(u,v_2)&
r_{u,t}(u)-a_{u,t}^{v_1,w}(u,v_2)-a_{u,t}^{v_2,w}(u,v_1)&
a_{u,t}^{v_2,w}(u,v_1)
\end{pmatrix}.
\]
\endgroup
The two row sums in this application are the complementary positive branch
masses $r_{t,u}(t)$ and $r_{u,t}(u)=1-r_{t,u}(t)$.  Identifying the resulting
$2\times2$ determinants with the corresponding $X$-variables gives
\begin{equation}
\frac16Q(S_{tu},S_{v_1v_2})
=\bigl(c^{v_1,w}(v_2)X_{t,u}^{v_1,w}
-c^{v_2,w}(v_1)X_{t,u}^{v_2,w}\bigr)^2.          \label{eq:edge-path-square}
\end{equation}
Moreover,
\begin{align*}
c^{v_1,w}(v_2)^2
={}&\sum_{\substack{u'\in N_{t+}(u)\\u''\in N_+(u)}}
q_{t,u}^{v_1,w}(u',u'',v_2,v_2)
+\sum_{\substack{t'\in N_{u+}(t)\\t''\in N_+(t)}}
q_{u,t}^{v_1,w}(t',t'',v_2,v_2),
\end{align*}
and the corresponding identity holds with $v_1$ and $v_2$ interchanged.
The mixed coefficient satisfies
\begin{align*}
c^{v_1,w}(v_2)c^{v_2,w}(v_1)
={}&\sum_{\substack{u'\in N_{t+}(u)\\u''\in N_+(u)}}
s_{t,u}^{v_1v_2}(u',u'')
+\sum_{\substack{t'\in N_{u+}(t)\\t''\in N_+(t)}}
s_{u,t}^{v_1v_2}(t',t'').
\end{align*}
Expanding \eqref{eq:edge-path-square} gives two pure-square terms and one mixed
term.  The identities above give both pure-square coefficients, whereas $P_2$
contains only one orientation of each path: one pure-square term corresponds
to the chosen orientation and the other to its reversal.  Summing over all
edges and path orbits therefore gives
\begin{align}
\frac16\sum_{\substack{tu\in E_1\\v_1wv_2\in E_2^2}}
Q(S_{tu},S_{v_1v_2})
&=\sum_{\substack{(t,u)\in D_1\\(v_1,w,v_2)\in P_2}}
\sum_{j\in\{1,2\}}
\sum_{\substack{u'\in N_{t+}(u)\\u''\in N_+(u)}}
q_{t,u}^{v_j,w}(u',u'',v_{\bar j},v_{\bar j})
(X_{t,u}^{v_j,w})^2                                        \notag\\
&-2\sum_{\substack{(t,u)\in D_1\\(v_1,w,v_2)\in P_2}}
\sum_{\substack{u'\in N_{t+}(u)\\u''\in N_+(u)}}
s_{t,u}^{v_1v_2}(u',u'')X_{t,u}^{v_1,w}X_{t,u}^{v_2,w}.
                                                               \label{eq:edge-path}
\end{align}

For a path $t_1ut_2$ of length $2$ in $T_1$ and an edge $vw$ in $T_2$,
\[
\frac16Q(S_{t_1t_2},S_{vw})=
\left|
\begin{matrix}
a_{t_1,u}^{v,w}(t_2,v)&a_{t_1,u}^{w,v}(t_2,w)\\
a_{t_2,u}^{v,w}(t_1,v)&a_{t_2,u}^{w,v}(t_1,w)
\end{matrix}
\right|^2.
\]
Apply Lemma~\ref{lem:matrix-identity} to
\begingroup
\small
\setlength{\arraycolsep}{12pt}
\renewcommand{\arraystretch}{1.12}
\[
\begin{pmatrix}
a_{t_1,u}^{v,w}(t_2,v)&
c^{v,w}(v)-a_{t_1,u}^{v,w}(t_2,v)-a_{t_2,u}^{v,w}(t_1,v)&
a_{t_2,u}^{v,w}(t_1,v)\\
a_{t_1,u}^{w,v}(t_2,w)&
c^{w,v}(w)-a_{t_1,u}^{w,v}(t_2,w)-a_{t_2,u}^{w,v}(t_1,w)&
a_{t_2,u}^{w,v}(t_1,w)
\end{pmatrix}.
\]
\endgroup
For $\alpha=1$ this yields
\begin{align*}
\left|
\begin{matrix}
a_{t_1,u}^{v,w}(t_2,v)&a_{t_2,u}^{v,w}(t_1,v)\\
a_{t_1,u}^{w,v}(t_2,w)&a_{t_2,u}^{w,v}(t_1,w)
\end{matrix}
\right|
={}&\frac{a_{t_2,u}^{v,w}(t_1,v)}{c^{v,w}(v)}X_{t_2,u}^{v,w}
-\frac{a_{t_1,u}^{v,w}(t_2,v)}{c^{v,w}(v)}X_{t_1,u}^{v,w},
\end{align*}
whereas $\alpha=0$ gives
\begin{align*}
\left|
\begin{matrix}
a_{t_1,u}^{v,w}(t_2,v)&a_{t_2,u}^{v,w}(t_1,v)\\
a_{t_1,u}^{w,v}(t_2,w)&a_{t_2,u}^{w,v}(t_1,w)
\end{matrix}
\right|
={}&\frac{a_{t_2,u}^{w,v}(t_1,w)}{c^{w,v}(w)}X_{t_2,u}^{v,w}
-\frac{a_{t_1,u}^{w,v}(t_2,w)}{c^{w,v}(w)}X_{t_1,u}^{v,w}.
\end{align*}
Multiplying these identities gives the square of the determinant on the
left.  Transposing its $2\times2$ matrix leaves the determinant unchanged
and identifies its square with $Q(S_{t_1t_2},S_{vw})/6$.  Changing
$d=(v,w)$ to $d=(w,v)$ changes both displayed $X$-variables by a factor
$-1$, so every quadratic monomial is unchanged.  We therefore obtain the
following formula for either
$d=(d_1,d_2)\in\{(w,v),(v,w)\}$, where
$X_{t_i,u}^{d}:=X_{t_i,u}^{d_1,d_2}$:
\begin{align*}
\frac16Q(S_{t_1t_2},S_{vw})
={}&\frac1{c^{v,w}(v)c^{w,v}(w)}
\Bigl[
a_{t_2,u}^{v,w}(t_1,v)a_{t_2,u}^{w,v}(t_1,w)(X_{t_2,u}^{d})^2\\
&\quad-\Bigl(
a_{t_2,u}^{v,w}(t_1,v)a_{t_1,u}^{w,v}(t_2,w)
+a_{t_1,u}^{v,w}(t_2,v)a_{t_2,u}^{w,v}(t_1,w)
\Bigr)X_{t_2,u}^{d}X_{t_1,u}^{d}\\
&\quad+a_{t_1,u}^{v,w}(t_2,v)a_{t_1,u}^{w,v}(t_2,w)(X_{t_1,u}^{d})^2
\Bigr].
\end{align*}
Take the convex combination of the cases $d=(w,v)$ and $d=(v,w)$ with
weights $c^{v,w}(v)$ and $c^{w,v}(w)$, respectively.  These two weights are
the normalized masses of the complementary sides of the edge split $S_{vw}$;
hence they are nonnegative and

\[
c^{v,w}(v)+c^{w,v}(w)=1.
\]
Substitution of the coefficient definitions expresses this convex combination
through the indexed $q$- and $b$-coefficients.
Since $P_1$ contains one representative of each
reversal pair, the sum below includes the pure-square terms associated with
both endpoint orientations.  Thus
\begin{align}
\frac16\sum_{\substack{t_1ut_2\in E_1^2\\vw\in E_2}}
Q(S_{t_1t_2},S_{vw})
&=\sum_{\substack{(t_1,u,t_2)\in P_1\\(v,w)\in D_2}}
\sum_{i\in\{1,2\}}
\sum_{\substack{w'\in N_{v+}(w)\\w''\in N_+(w)}}
q_{t_i,u}^{v,w}(t_{\bar i},t_{\bar i},w',w'')
(X_{t_i,u}^{v,w})^2                                        \notag\\
&-2\sum_{\substack{(t_1,u,t_2)\in P_1\\(v,w)\in D_2}}
\sum_{\substack{w'\in N_{v+}(w)\\w''\in N_+(w)}}
b_{t_1t_2}^{v,w}(w',w'')X_{t_2,u}^{v,w}X_{t_1,u}^{v,w}.
                                                               \label{eq:path-edge}
\end{align}

If $t_1ut_2$ and $v_1wv_2$ are paths of length $2$ in $T_1$ and $T_2$,
respectively, then
\[
\frac16Q(S_{t_1t_2},S_{v_1v_2})=
\left|
\begin{matrix}
a_{t_1,u}^{v_1,w}(t_2,v_2)&a_{t_1,u}^{v_2,w}(t_2,v_1)\\
a_{t_2,u}^{v_1,w}(t_1,v_2)&a_{t_2,u}^{v_2,w}(t_1,v_1)
\end{matrix}
\right|^2.
\]
Apply Lemma~\ref{lem:matrix-identity} to
\begingroup
\small
\setlength{\arraycolsep}{12pt}
\renewcommand{\arraystretch}{1.12}
\[
\begin{pmatrix}
a_{t_1,u}^{v_1,w}(t_2,v_2)&
c^{v_1,w}(v_2)-a_{t_1,u}^{v_1,w}(t_2,v_2)-a_{t_2,u}^{v_1,w}(t_1,v_2)&
a_{t_2,u}^{v_1,w}(t_1,v_2)\\
a_{t_1,u}^{v_2,w}(t_2,v_1)&
c^{v_2,w}(v_1)-a_{t_1,u}^{v_2,w}(t_2,v_1)-a_{t_2,u}^{v_2,w}(t_1,v_1)&
a_{t_2,u}^{v_2,w}(t_1,v_1)
\end{pmatrix}.
\]
\endgroup
For $\alpha=1$, repeated use of the determinant identity used to derive
\eqref{eq:edge-path-square} gives
\begin{align*}
\left|
\begin{matrix}
a_{t_1,u}^{v_1,w}(t_2,v_2)&a_{t_2,u}^{v_1,w}(t_1,v_2)\\
a_{t_1,u}^{v_2,w}(t_2,v_1)&a_{t_2,u}^{v_2,w}(t_1,v_1)
\end{matrix}
\right|
&=\frac{a_{t_2,u}^{v_1,w}(t_1,v_2)}{c^{v_1,w}(v_2)}
\Bigl(c^{v_2,w}(v_1)X_{t_2,u}^{v_2,w}
-c^{v_1,w}(v_2)X_{t_2,u}^{v_1,w}\Bigr)                         \\
&-\frac{a_{t_1,u}^{v_1,w}(t_2,v_2)}{c^{v_1,w}(v_2)}
\Bigl(c^{v_2,w}(v_1)X_{t_1,u}^{v_2,w}
-c^{v_1,w}(v_2)X_{t_1,u}^{v_1,w}\Bigr).
\end{align*}
Correspondingly, $\alpha=0$ yields
\begin{align*}
\left|
\begin{matrix}
a_{t_1,u}^{v_1,w}(t_2,v_2)&a_{t_2,u}^{v_1,w}(t_1,v_2)\\
a_{t_1,u}^{v_2,w}(t_2,v_1)&a_{t_2,u}^{v_2,w}(t_1,v_1)
\end{matrix}
\right|
&=\frac{a_{t_2,u}^{v_2,w}(t_1,v_1)}{c^{v_2,w}(v_1)}
\Bigl(c^{v_2,w}(v_1)X_{t_2,u}^{v_2,w}
-c^{v_1,w}(v_2)X_{t_2,u}^{v_1,w}\Bigr)                         \\
&-\frac{a_{t_1,u}^{v_2,w}(t_2,v_1)}{c^{v_2,w}(v_1)}
\Bigl(c^{v_2,w}(v_1)X_{t_1,u}^{v_2,w}
-c^{v_1,w}(v_2)X_{t_1,u}^{v_1,w}\Bigr).
\end{align*}
Multiplying the two equations gives the square of the determinant on the
left.  Transposition leaves this determinant unchanged and identifies its
square with $Q(S_{t_1t_2},S_{v_1v_2})/6$.  Substituting the definitions of the
indexed $q$-, $b$-, $s$-, and $p$-coefficients gives
\begin{align}
\frac16Q(S_{t_1t_2},S_{v_1v_2})
={}&\sum_{i,j\in\{1,2\}}
q_{t_i,u}^{v_j,w}(t_{\bar i},t_{\bar i},v_{\bar j},v_{\bar j})
(X_{t_i,u}^{v_j,w})^2                                         \notag\\
&-2\sum_{j\in\{1,2\}}
b_{t_1t_2}^{v_j,w}(v_{\bar j},v_{\bar j})
X_{t_1,u}^{v_j,w}X_{t_2,u}^{v_j,w}                             \notag\\
&-2\sum_{i\in\{1,2\}}
s_{t_i,u}^{v_1v_2}(t_{\bar i},t_{\bar i})
X_{t_i,u}^{v_1,w}X_{t_i,u}^{v_2,w}                             \notag\\
&+2p_{t_1t_2}^{v_1v_2}
\Bigl(X_{t_1,u}^{v_1,w}X_{t_2,u}^{v_2,w}
+X_{t_1,u}^{v_2,w}X_{t_2,u}^{v_1,w}\Bigr).             \label{eq:path-path}
\end{align}

Set
\[
K_{t,u}:=N_{t+}(u)\times N_+(u),
\qquad
\Delta_{t,u}:=\{(u',u'):u'\in N_t(u)\}.
\]
Then $K_{t,u}\setminus\Delta_{t,u}=I'_{t,u}$.  For the pure-square terms
whose coefficients belong to the $q$-family, inclusion--exclusion gives
\[
\mathbf 1_{K_{t,u}\times K_{v,w}}
-\mathbf 1_{K_{t,u}\times\Delta_{v,w}}
-\mathbf 1_{\Delta_{t,u}\times K_{v,w}}
+\mathbf 1_{\Delta_{t,u}\times\Delta_{v,w}}
=\mathbf 1_{I'_{t,u}\times I'_{v,w}}.
\]
For the $s$-mixed terms, the corresponding row-coordinate identity is
$\mathbf 1_{K_{t,u}}-\mathbf 1_{\Delta_{t,u}}=\mathbf 1_{I'_{t,u}}$; for
the $b$-mixed terms, the column-coordinate identity is the analogous formula
with $(v,w)$ in place of $(t,u)$.  The mixed terms carrying a
$p$-coefficient occur only in the path--path contribution.  Substituting
\eqref{eq:edge-edge},
\eqref{eq:edge-path}, \eqref{eq:path-edge}, and \eqref{eq:path-path} into
Theorem~\ref{thm:inclusion-exclusion} and collecting like monomials gives
\begin{align}
\frac16Q(\mathcal T_1,\mathcal T_2)
={}&\sum_{\substack{(t,u)\in D_1\\(v,w)\in D_2}}
\sum_{\substack{(u',u'')\in I'_{t,u}\\(w',w'')\in I'_{v,w}}}
q_{t,u}^{v,w}(u',u'',w',w'')(X_{t,u}^{v,w})^2                 \notag\\
&+2\sum_{\substack{(t,u)\in D_1\\(v_1,w,v_2)\in P_2}}
\sum_{(u',u'')\in I'_{t,u}}
s_{t,u}^{v_1v_2}(u',u'')X_{t,u}^{v_1,w}X_{t,u}^{v_2,w}        \notag\\
&+2\sum_{\substack{(t_1,u,t_2)\in P_1\\(v,w)\in D_2}}
\sum_{(w',w'')\in I'_{v,w}}
b_{t_1t_2}^{v,w}(w',w'')X_{t_1,u}^{v,w}X_{t_2,u}^{v,w}        \notag\\
&+2\sum_{\substack{(t_1,u,t_2)\in P_1\\(v_1,w,v_2)\in P_2}}
p_{t_1t_2}^{v_1v_2}
\Bigl(X_{t_1,u}^{v_1,w}X_{t_2,u}^{v_2,w}
+X_{t_1,u}^{v_2,w}X_{t_2,u}^{v_1,w}\Bigr).            \label{eq:raw-quadratic}
\end{align}

We next compare this quadratic polynomial with the expansion of the
right-hand side of Theorem~\ref{thm:sum-of-squares}.  For each directed edge
put
\[
B_{t,u}:=I'_{t,u}\setminus I_{t,u}
=\{(u',t):u'\in N_t(u)\},
\qquad I'_{t,u}=I_{t,u}\mathbin{\dot\cup}B_{t,u},
\]
and define $B_{v,w}$ analogously.  The definitions of $b$, $s$, and $p$ give
the following indexed boundary identities:
\begin{equation*}
\begin{aligned}
q_{t,u}^{v,w}(u',u'',w',v)&=s_{t,u}^{vw'}(u',u'')
&&\text{if }(w',v)\in B_{v,w},\\
q_{t,u}^{v,w}(u',t,w',w'')&=b_{tu'}^{v,w}(w',w'')
&&\text{if }(u',t)\in B_{t,u},\\
q_{t,u}^{v,w}(u',t,w',v)&=p_{tu'}^{vw'}
&&\text{if }(u',t)\in B_{t,u}\text{ and }(w',v)\in B_{v,w},\\
s_{t,u}^{vw'}(u',t)&=p_{tu'}^{vw'}
&&\text{if }(u',t)\in B_{t,u}\text{ and }(w',v)\in B_{v,w},\\
b_{tu'}^{v,w}(w',v)&=p_{tu'}^{vw'}
&&\text{if }(u',t)\in B_{t,u}\text{ and }(w',v)\in B_{v,w}.
\end{aligned}
\end{equation*}
These equalities concern only the coefficients written above at the stated
indices; they do not assert equality of the complete $q$, $s$, $b$, or $p$
coefficient families.
The disjoint decomposition
\[
I'_{t,u}\times I'_{v,w}
=(I_{t,u}\times I_{v,w})
\mathbin{\dot\cup}(I_{t,u}\times B_{v,w})
\mathbin{\dot\cup}(B_{t,u}\times I_{v,w})
\mathbin{\dot\cup}(B_{t,u}\times B_{v,w})
\]
and the preceding identities give the following correspondence.  In the
first factor of each Cartesian product, $I$ and $B$ mean membership in
$I_{t,u}$ and $B_{t,u}$; in the second factor, they mean membership in
$I_{v,w}$ and $B_{v,w}$.
Within an edge--path term write $X_j:=X_{t,u}^{v_j,w}$; within a path--edge
term write $X_i:=X_{t_i,u}^{v,w}$; and within a path--path term write
$X_{ij}:=X_{t_i,u}^{v_j,w}$.  Thus $i\in\{1,2\}$ selects the endpoint
$t_i$, equivalently the directed edge $(t_i,u)$, of the fixed first-tree path,
and $j\in\{1,2\}$ selects $v_j$, equivalently $(v_j,w)$, of the fixed
second-tree path.  In an $s$-row of the table, a lone $I$ or $B$ records
membership of the first-tree pair $(u',u'')$ in $I_{t,u}$ or $B_{t,u}$; in a
$b$-row it records membership of the second-tree pair $(w',w'')$ in
$I_{v,w}$ or $B_{v,w}$.  The
bare letters $q,s,b,p$
in the table below are abbreviations for the previously defined coefficient
families carrying the same letters,
\[
q_{t,u}^{v,w}(u',u'',w',w''),\qquad
s_{t,u}^{v_1v_2}(u',u''),\qquad
b_{t_1t_2}^{v,w}(w',w''),\qquad
p_{t_1t_2}^{v_1v_2};
\]
they are not new quantities.  Indices fixed by the row of the table and by
the surrounding sums are suppressed only to keep the table readable.  If
$\lambda$ denotes the relevant coefficient, a \emph{pure-square term} means
$\lambda X_\alpha^2$, whereas a \emph{mixed term} means
$2\lambda X_\alpha X_\beta$ with $\alpha\ne\beta$.  For the
$p$-family, view the variables as the $2\times2$ array $(X_{ij})$.  Thus an
$s$-mixed term joins entries in the same row of this array, a $b$-mixed term
joins entries in the same column, and the explicit $p$-term in
\eqref{eq:raw-quadratic} consists of the two opposite-corner mixed terms:
\[
\begin{array}{c|c}
\text{term in \eqref{eq:raw-quadratic}}&
\text{corresponding term in Theorem~\ref{thm:sum-of-squares}}\\ \hline
qX^2, I\times I&\text{pure-square term in the }q\text{-sum}\\
qX^2, I\times B&\text{pure-square term in the }s\text{-sum (first boundary identity)}\\
qX^2, B\times I&\text{pure-square term in the }b\text{-sum (second boundary identity)}\\
qX^2, B\times B&\text{pure-square term in the }p\text{-sum (third boundary identity)}\\
2sX_jX_{\bar j}, I&\text{mixed term in the }s\text{-sum}\\
2sX_jX_{\bar j}, B&\text{same-row mixed term in the }p\text{-sum (fourth identity)}\\
2bX_iX_{\bar i}, I&\text{mixed term in the }b\text{-sum}\\
2bX_iX_{\bar i}, B&\text{same-column mixed term in the }p\text{-sum (fifth identity)}\\
2p(X_{11}X_{22}+X_{12}X_{21})&
\text{opposite-corner mixed terms in the }p\text{-sum}
\end{array}
\]
The four Cartesian products in the displayed decomposition are disjoint.
Moreover,
the derivations of \eqref{eq:edge-path}, \eqref{eq:path-edge}, and
\eqref{eq:path-path} already combine the contributions from both orientations
of each reversal orbit, whereas $P_1$ and $P_2$ select one representative of
each orbit.  Hence every indexed term in \eqref{eq:raw-quadratic} occurs once
in the expansion of the four sums of squares.  The two expressions are equal,
which proves the theorem.
\end{proof}

\begin{lemma}[Partial-quartet reconstruction]\label{lem:partial-reconstruction}
Let $\mathcal T$ be an $X$-tree and let $A\mid(X\setminus A)$ be a full
split, so $A\ne\varnothing$ and $X\setminus A\ne\varnothing$.  We say that
an edge $e$ represents this split when the unordered split $S_e$ equals
$A\mid(X\setminus A)$.  The split is represented by an edge of $\mathcal T$
if and only if every partial quartet displayed by it is displayed by
$\mathcal T$.  Consequently the three systems $Q_2(\mathcal T)$,
$Q_3(\mathcal T)$, and $Q_4(\mathcal T)$ determine $\mathcal T$ up to a
label-preserving graph isomorphism.

Moreover, if $|X|\ge4$, an $X$-tree that displays $xx\mid yz$ for every
three distinct taxa $x,y,z$ and displays a full quartet on every four-set is
binary phylogenetic.
\end{lemma}

\begin{proof}
Only the reverse implication in the first assertion needs proof.  Choose
$b_2\in X\setminus A$ and root the graph at its labelled vertex $\phi(b_2)$.
For a vertex, its rooted cluster is the union of the label fibres in its
descendant subtree.  Among the rooted clusters containing $A$, choose one of
minimum cardinality, say $C$ with root vertex $v$.
If $C=A$, the edge above $v$ represents the proposed split.  If $C\ne A$,
minimality supplies two possibly equal taxa $a_1,a_2\in A$ such that the path
between $\phi(a_1)$ and $\phi(a_2)$ contains $v$: either $v$ itself carries a
taxon of $A$, or $A$ meets
at least two child branches of $v$, since otherwise a smaller child cluster
would contain $A$.  Choose $b_1\in C\setminus A$; the path from
$\phi(b_2)$ to $\phi(b_1)$ also contains $v$.  The partial quartet
$a_1a_2\mid b_1b_2$ is displayed by $A\mid(X\setminus A)$, but no tree edge
can display it: such an edge would put $v$ simultaneously on both sides of
the cut, because $v$ lies on both within-side paths.  This contradiction
proves the criterion.

If two trees have the same three partial-quartet systems, apply the criterion
first to every edge split of the first tree using the second tree, and then
vice versa, to obtain the same full edge-split system.  After rooting each
graph at the labelled vertex of the same taxon, the sides away from the root
form the same
laminar family of rooted clusters.  Inclusion in this family recovers the
parent--child relation, and the label fibre at a vertex is its cluster minus
the disjoint union of its child clusters.  This reconstructs the required
label-preserving isomorphism.

For the last assertion, the three-taxon condition first forces distinct taxa
to occupy distinct vertices.  If a labelled vertex were not a leaf, choose
taxa in two different branches at that vertex; then the required split
$xx\mid yz$ could not be displayed.  Hence every taxon labels a distinct
leaf and every leaf is labelled; moreover, there are no degree-$2$ vertices.
Finally a vertex of degree at
least $4$ has four labelled branches; choosing one taxon in each produces a
four-set with no displayed $2$--$2$ split.  Thus every nonleaf vertex has
degree $3$.
\end{proof}

\begin{lemma}[Vertex--side independence bridge]
\label{lem:vertex-side-bridge}
Let $U$ be uniform on $X$ and put $Z_i=\phi_i(U)$ for $i\in\{1,2\}$.
The trees $\mathcal T_1,\mathcal T_2$ are independent if and only if the
finite-valued random variables $Z_1,Z_2$ are independent.  If they are not
independent, there exist vertices $u\in V_1$, $w\in V_2$ with
$P(Z_1=u)>0$ and $P(Z_2=w)>0$, and
incident directed edges $(t,u)\in D_1$, $(v,w)\in D_2$ such that
\[
Q(S_{tu},S_{vw})>0,
\qquad r_{t,u}(u)>0,
\qquad c^{v,w}(w)>0.
\]
\end{lemma}

\begin{proof}
Every side of an edge split is a union of vertex events, so independence of
$Z_1,Z_2$ implies independence of all edge-split pairs.  Conversely, root
each tree.  For a vertex $z$ of tree $i$, let $C_z^{(i)}$ be the union of
the label fibres in the descendant subtree rooted at $z$, with
$C_z^{(i)}=X$ at the root.  For $C\subseteq X$, write
$\mathbf 1_C:=\mathbf 1_{\{U\in C\}}$.  If $y$ ranges over the
children of $z$, then
\[
\mathbf 1_{\{Z_i=z\}}=\mathbf 1_{C_z^{(i)}}
-\sum_{y\text{ child of }z}\mathbf 1_{C_y^{(i)}}.
\]
Every non-root cluster is a side of an edge split.  If all edge-split pairs
are independent, every cross-covariance between their side indicators is
zero.  Bilinearity of covariance therefore makes every covariance between a
$Z_1$-singleton indicator and a $Z_2$-singleton indicator zero.  Since these
are Bernoulli indicators, this says
\[
P(Z_1=z,Z_2=z')=P(Z_1=z)P(Z_2=z')
\]
for every pair of vertices $z,z'$, which is exactly independence of
$Z_1,Z_2$.

Now suppose the trees are not independent.  Then some vertices $u,w$ with
positive marginal probabilities satisfy
$\operatorname{Cov}(\mathbf 1_{\{Z_1=u\}},
\mathbf 1_{\{Z_2=w\}})\ne0$.  Root the trees at $u,w$.  Neither random
variable is constant, so both roots have incident edges, and
\[
\mathbf 1_{\{Z_1=u\}}=1-\sum_{t\in N(u)}\mathbf 1_{C_t^{(1)}},
\qquad
\mathbf 1_{\{Z_2=w\}}=1-\sum_{v\in N(w)}\mathbf 1_{C_v^{(2)}}.
\]
Bilinearity supplies $t\in N(u)$ and $v\in N(w)$ for which the two branch
indicators have nonzero covariance.  Their $2\times2$ intersection table has
nonzero determinant, hence $Q(S_{tu},S_{vw})>0$.  Orienting both edges toward
the roots gives
$r_{t,u}(u)=P(Z_1=u)>0$ and $c^{v,w}(w)=P(Z_2=w)>0$.
\end{proof}

\begin{corollary}\label{cor:Q-bounds}
For two $X$-trees $\mathcal T_1$ and $\mathcal T_2$, with $n=|X|$,
\[
0\le Q(\mathcal T_1,\mathcal T_2)\le1-\frac{4n-3}{n^3}.
\]
The lower bound is attained if and only if $\mathcal T_1$ and
$\mathcal T_2$ are independent.  If $n\ge4$, then the upper bound is
attained if and only if the trees are binary phylogenetic and are
label-preservingly isomorphic.
\end{corollary}

\begin{proof}
The lower bound follows from Theorem~\ref{thm:sum-of-squares}, and the
quartet covariance between independent $X$-trees is zero.  Conversely,
Lemma~\ref{lem:vertex-side-bridge} supplies an edge-split pair with positive
quartet covariance for
which $r_{t,u}(u)>0$ and $c^{v,w}(w)>0$.
Lemma~\ref{lem:coefficient-lower-bound} below then gives
\[
Q(\mathcal T_1,\mathcal T_2)
\ge\tfrac12c^{v,w}(w)r_{t,u}(u)Q(S_{tu},S_{vw})>0.
\]

The probability that some taxon is drawn at least three times among the four
draws is $(4n-3)/n^3$, and every such ordered quadruple contributes zero to
$Q(\mathcal T_1,\mathcal T_2)$.  Every other ordered
quadruple contributes at most $1$, proving the upper bound.  Equality means
that every ordered quadruple outside the event just excluded contributes $1$:
the two trees display the same
partial quartet on every support of size $2$ or $3$, and the same full
quartet on every four-set.  If $n\ge4$, the final assertion of
Lemma~\ref{lem:partial-reconstruction} makes both trees binary
phylogenetic, and its reconstruction assertion makes them
label-preservingly isomorphic.  Conversely, a binary phylogenetic tree
displays every non-full partial quartet and exactly one full quartet on each
four-set.  Two label-preservingly isomorphic such trees therefore make every
eligible ordered quadruple contribute $1$, so equality holds.
\end{proof}

To relate quartet covariance to quartet distance, we record the following
counting identity.

\begin{lemma}\label{lem:Q-counts}
For two $X$-trees $\mathcal T_1$ and $\mathcal T_2$,
\begin{align*}
Q(\mathcal T_1,\mathcal T_2)
=\frac1{|X|^4}\Bigl(&24|Q_4(\mathcal T_1)\cap Q_4(\mathcal T_2)|
+12|Q_3(\mathcal T_1)\cap Q_3(\mathcal T_2)|\\
&+6|Q_2(\mathcal T_1)\cap Q_2(\mathcal T_2)|
-12d_Q(\mathcal T_1,\mathcal T_2)\Bigr).
\end{align*}
\end{lemma}

\begin{proof}
For a partial quartet $x_1x_1\mid x_2x_2$, there are $6$ ordered
$4$-tuples containing $x_1$ and $x_2$ twice.  For
$x_1x_1\mid x_2x_3$, there are $12$ ordered $4$-tuples containing
$x_1$ twice and $x_2,x_3$ once.  For a $4$-set there are $24$
ordered $4$-tuples containing each element once.  A common displayed object
contributes $+1$ for each ordering, whereas a four-set counted by $d_Q$
contributes $-1/2$ for each of its $24$ orderings.  Every ordered quadruple
has probability $|X|^{-4}$, which proves the formula.
\end{proof}

The $2/3$-conjecture is now a direct consequence of
Theorem~\ref{thm:sum-of-squares}.

\begin{corollary}\label{cor:two-thirds}
For two binary phylogenetic $X$-trees $\mathcal T_1$ and $\mathcal T_2$,
\[
d_Q(\mathcal T_1,\mathcal T_2)
\le\frac23\binom{|X|}{4}+\binom{|X|}{3}+\frac16\binom{|X|}{2}.
\]
\end{corollary}

\begin{proof}
By Lemma~\ref{lem:Q-counts},
\[
Q(\mathcal T_1,\mathcal T_2)=\frac1{|X|^4}\Biggl(
24\left(\binom{|X|}{4}-d_Q(\mathcal T_1,\mathcal T_2)\right)
-12d_Q(\mathcal T_1,\mathcal T_2)
+36\binom{|X|}{3}+6\binom{|X|}{2}\Biggr),
\]
because every $4$-set induces one quartet in each tree, while all
non-full partial quartets are displayed by both binary phylogenetic
trees.  The nonnegativity of $Q$ yields the claimed inequality.
\end{proof}

Pachter~\cite{Pachter2026} also proves the same asymptotic $2/3$
conclusion by a common-root planarisation argument and a five-leaf identity.
Our proof instead derives the displayed finite-$|X|$ inequality from
the nonnegative sum-of-squares representation in
Theorem~\ref{thm:sum-of-squares}.

For $|X|\ge2$ this explicit bound is not sharp, because two phylogenetic
trees are never independent.  For any taxon $x$, compare the two pendant
edge splits having singleton side $\{x\}$.  Their $2\times2$ intersection
table has determinant $|X|-1\ne0$, so their split covariance is positive.
On the other hand, the difference
between the upper bound and the infinite family of tree pairs from
\cite{ChorErdosKomornik2019} is only $O(|X|^2\log|X|)$.

We end this section with a lower bound on the coefficient multiplying
$(X_{t,u}^{v,w})^2$ in the $q$-square family, namely the first nonnegative
sum in Theorem~\ref{thm:sum-of-squares}.

\begin{lemma}\label{lem:coefficient-lower-bound}
For two $X$-trees $\mathcal T_1$ and $\mathcal T_2$, and directed edges
$(t,u)\in D_1$, $(v,w)\in D_2$,
\[
\sum_{\substack{(u',u'')\in I_{t,u}\\(w',w'')\in I_{v,w}}}
q_{t,u}^{v,w}(u',u'',w',w'')
\ge\frac12c^{v,w}(w)r_{t,u}(u).
\]
In particular,
\[
Q(\mathcal T_1,\mathcal T_2)
\ge\frac12\sum_{(t,u)\in D_1}\sum_{(v,w)\in D_2}
c^{v,w}(w)r_{t,u}(u)Q(S_{tu},S_{vw}).
\]
\end{lemma}

\begin{proof}
Fix $(t,u)\in D_1$ and $(v,w)\in D_2$, and abbreviate
\[
a(x,y):=a_{t,u}^{v,w}(x,y),\qquad
c(y):=c^{v,w}(y),\qquad r(x):=r_{t,u}(x).
\]
For $w'\in N_{v+}(w)$, put
\[
J(w'):=\{w''\in N_+(w):w'=w\text{ or }w''\notin\{v,w'\}\},
\qquad
\sigma(w'):=\sum_{w''\in J(w')}c(w'').
\]
The pair $(w',w)$ belongs to $I_{v,w}$ for every
$w'\in N_{v+}(w)$, so $\sigma(w')\ge c(w)$.  If $w'=w$, then every
$w''\in N_+(w)$ is allowed, whence $J(w)=N_+(w)$ and $\sigma(w)=1$.

All $q$-coefficients are nonnegative.  Since $u'=u$ makes every
$u''\in N_+(u)$ admissible in $I_{t,u}$, retaining only those terms gives
\[
\sum_{(u',u'')\in I_{t,u}}
\sum_{(w',w'')\in I_{v,w}}q_{t,u}^{v,w}(u',u'',w',w'')
\ge
\sum_{w'\in N_{v+}(w)}\sum_{w''\in J(w')}
\sum_{u''\in N_+(u)}q_{t,u}^{v,w}(u,u'',w',w'').
\]
The row blocks indexed by $N_+(u)$ partition $X$, so
\[
\sum_{u''\in N_+(u)}a(u'',v)=c(v),
\qquad
\sum_{u''\in N_+(u)}a(u'',w')=c(w').
\]
The definition of the coefficient family $q$ turns the retained triple sum into
\begin{equation}
\frac12\sum_{w'\in N_{v+}(w)}\sigma(w')
\left(a(u,w')+c(w')\frac{a(u,v)}{c(v)}\right).
\label{eq:lemma3-retained-sum}
\end{equation}
In $T_2$, the component of $T_2-vw$ containing $v$ contains at least one
label, so $0<c(v)\le1$.  The column blocks indexed
by $N_{v+}(w)$ are all column blocks except that indexed by $v$.  Hence
\begin{align*}
\sum_{w'\in N_{v+}(w)}\sigma(w')a(u,w')
&\ge c(w)\bigl(r(u)-a(u,v)\bigr),\\
\sum_{w'\in N_{v+}(w)}\sigma(w')c(w')\frac{a(u,v)}{c(v)}
&\ge \sigma(w)c(w)\frac{a(u,v)}{c(v)}
\ge c(w)a(u,v).
\end{align*}
Substitution into \eqref{eq:lemma3-retained-sum} proves the first claim.

For the final assertion, retain only the $q$-square family, namely the first
nonnegative sum in Theorem~\ref{thm:sum-of-squares}, and use
$Q(S_{tu},S_{vw})=6(X_{t,u}^{v,w})^2$:
\begin{align*}
\frac16Q(\mathcal T_1,\mathcal T_2)
&\ge\sum_{(t,u)\in D_1}\sum_{(v,w)\in D_2}
\left(
\sum_{(u',u'')\in I_{t,u}}
\sum_{(w',w'')\in I_{v,w}}
q_{t,u}^{v,w}(u',u'',w',w'')
\right)
(X_{t,u}^{v,w})^2\\
&\ge\frac1{12}\sum_{(t,u)\in D_1}\sum_{(v,w)\in D_2}
c^{v,w}(w)r_{t,u}(u)Q(S_{tu},S_{vw}).
\end{align*}
Multiplication by $6$ proves the stated bound.
\end{proof}

\section{Statistical consequences and dependence measures}
\label{sec:statistical-consequences}

\subsection{Normalised quartet dependence measures}

The quartet covariance is a similarity measure for $X$-trees related to the
quartet distance.  Nevertheless, quartet covariance never reaches
$1$, even when the compared trees are identical.  A low value may reflect
incompatibility or lack of resolution.  We therefore consider two
normalisations.

For two non-trivial $X$-trees, define the \emph{quartet compatibility index}
$c_Q(\mathcal T_1,\mathcal T_2)$ to be the quartet covariance divided by the
sum of (i) the probability of drawing a partial quartet displayed by both
trees and (ii) half the probability that the draw induces distinct full
quartets in the two trees.  Its
denominator can be computed from Theorem~\ref{thm:inclusion-exclusion}
by replacing the determinant in $Q(S_1,S_2)$ with the permanent of the
same matrix.  Here
\[
\operatorname{per}\!\begin{pmatrix}a&b\\c&d\end{pmatrix}:=ad+bc,
\]
in contrast to the determinant $ad-bc$.  Equivalently,
\[
c_Q(\mathcal T_1,\mathcal T_2)=
\frac{
24|Q_4(\mathcal T_1)\cap Q_4(\mathcal T_2)|
\mkern-2mu+\mkern-2mu12|Q_3(\mathcal T_1)\cap Q_3(\mathcal T_2)|
\mkern-2mu+\mkern-2mu6|Q_2(\mathcal T_1)\cap Q_2(\mathcal T_2)|
\mkern-2mu-\mkern-2mu12d_Q(\mathcal T_1,\mathcal T_2)
}{
24|Q_4(\mathcal T_1)\cap Q_4(\mathcal T_2)|
\mkern-2mu+\mkern-2mu12|Q_3(\mathcal T_1)\cap Q_3(\mathcal T_2)|
\mkern-2mu+\mkern-2mu6|Q_2(\mathcal T_1)\cap Q_2(\mathcal T_2)|
\mkern-2mu+\mkern-2mu12d_Q(\mathcal T_1,\mathcal T_2)
}.
\]
This denominator is strictly positive.  Indeed, choose one edge split
$A\mid B$ from $\mathcal T_1$ and one edge split $C\mid D$ from
$\mathcal T_2$.  Since all four sides are nonempty, either both
$A\cap C$ and $B\cap D$ are nonempty, or both $A\cap D$ and $B\cap C$ are
nonempty.  Choosing one taxon from each of the corresponding two
intersections gives a partial quartet in
$Q_2(\mathcal T_1)\cap Q_2(\mathcal T_2)$.

Two full splits $A\mid B$ and $C\mid D$ are called \emph{compatible} if
at least one of $A\cap C$, $A\cap D$, $B\cap C$, and $B\cap D$ is empty.
A split system is compatible if every pair of its splits is compatible.

\begin{proposition}\label{prop:compatibility}
For two non-trivial $X$-trees $\mathcal T_1$ and $\mathcal T_2$,
$0\le c_Q(\mathcal T_1,\mathcal T_2)\le1$.  Moreover,
$c_Q(\mathcal T_1,\mathcal T_2)=0$ if and only if the trees are
independent.  Furthermore, $c_Q(\mathcal T_1,\mathcal T_2)=1$ if and only if
the union of their edge-split systems is compatible.
\end{proposition}

\begin{proof}
By the preceding definition, compatible splits cannot display
conflicting partial quartets.  Conversely, if two edge splits are
incompatible, choosing one taxon from each of their four nonempty
cross-intersections produces a conflicting quartet.  Hence the union of the
two edge-split systems is compatible exactly when
$d_Q(\mathcal T_1,\mathcal T_2)=0$.  The defining fraction for $c_Q$ is equal
to $1$ exactly when this conflicting term
vanishes.  The zero statement follows from Corollary~\ref{cor:Q-bounds} and
the strict positivity of the denominator established above.
\end{proof}

To measure similarity rather than compatibility, define the
\emph{quartet correlation}
\[
\rho_Q(\mathcal T_1,\mathcal T_2):=
\frac{Q(\mathcal T_1,\mathcal T_2)}
{\sqrt{Q(\mathcal T_1,\mathcal T_1)Q(\mathcal T_2,\mathcal T_2)}}.
\]

\begin{proposition}\label{prop:correlation}
For two non-trivial $X$-trees $\mathcal T_1$ and $\mathcal T_2$,
$0\le\rho_Q(\mathcal T_1,\mathcal T_2)
\le c_Q(\mathcal T_1,\mathcal T_2)$.
Furthermore, $\rho_Q=0$ if and only if the trees are independent, while
$0<\rho_Q=c_Q$ if and only if the trees are label-preservingly isomorphic
and $Q(\mathcal T_1,\mathcal T_1)>0$.  In the latter case,
$\rho_Q=c_Q=1$.
\end{proposition}

\begin{proof}
For $i\in\{1,2\}$,
\[
Q(\mathcal T_i,\mathcal T_i)=\frac1{|X|^4}
\left(24|Q_4(\mathcal T_i)|+12|Q_3(\mathcal T_i)|+6|Q_2(\mathcal T_i)|\right).
\]
Each self-comparison contains every common displayed partial quartet.  For
each $i\in\{1,2\}$, the full-quartet system $Q_4(\mathcal T_i)$ also contains
one quartet for each of the $d_Q$ four-sets on which the two trees disagree.
Thus, in particular,
\[
|Q_4(\mathcal T_i)|\ge
|Q_4(\mathcal T_1)\cap Q_4(\mathcal T_2)|+d_Q(\mathcal T_1,\mathcal T_2),
\]
with the analogous common-system containments for $Q_2$ and $Q_3$.
Because each tree is non-trivial, it displays at least one support-$2$ partial
quartet; hence both self covariances are positive.  Comparing the preceding
self sums with the numerator and
denominator of $c_Q$ gives, for $i\in\{1,2\}$,
\[
\frac{Q(\mathcal T_1,\mathcal T_2)}{Q(\mathcal T_i,\mathcal T_i)}
\le c_Q(\mathcal T_1,\mathcal T_2).
\]
Multiplying the two inequalities and taking nonnegative square roots gives
$\rho_Q\le c_Q$.  Moreover, $\rho_Q=0$ is now equivalent to
$Q(\mathcal T_1,\mathcal T_2)=0$, and
hence to independence by Corollary~\ref{cor:Q-bounds}.

Suppose now that $0<\rho_Q=c_Q$.  Then the two inequalities above are both
equalities and their common numerator is positive.  Put
\begin{align*}
A&:=24|Q_4(\mathcal T_1)\cap Q_4(\mathcal T_2)|
+12|Q_3(\mathcal T_1)\cap Q_3(\mathcal T_2)|
+6|Q_2(\mathcal T_1)\cap Q_2(\mathcal T_2)|,\\
B&:=12d_Q(\mathcal T_1,\mathcal T_2).
\end{align*}
Thus $A$ is the weighted number of common displayed objects.  The denominator
of $c_Q$ is $A+B$, whereas, for $i\in\{1,2\}$, the weighted self-comparison
count $|X|^4Q(\mathcal T_i,\mathcal T_i)$ is $A+2B$ plus the nonnegative weighted
count of the remaining partial quartets displayed only by that tree, excluding the
disagreement quartets already included in $2B$.  Equality of the positive
ratios therefore forces $B=0$ and both exclusive counts to vanish.  Thus the
two trees have the same displayed partial quartets of support $2$, $3$, and
$4$.  Those systems determine the $X$-tree up to a graph isomorphism
preserving every taxon label by Lemma~\ref{lem:partial-reconstruction}.
Conversely, a label-preserving isomorphism makes the two displayed systems and
both self covariances equal; if the self covariance is positive, both
normalized indices are $1$.  This proves the stated equality characterization.
\end{proof}

\subsection{Tree-metric distance covariance}

One can also weight partial quartets by their strength of support.  This
is especially useful for trees equipped with edge lengths.  An
\emph{edge-weighted $X$-tree} is a triple $\mathcal T=(T,\phi,l)$, where
$(T,\phi)$ is an $X$-tree and $l:E(T)\to(0,\infty)$ assigns a positive
length to each edge.  An unweighted tree is recovered by taking
$l(e)=1$ for all $e$.  The \emph{length} of a partial quartet
$x_1x_2\mid x_3x_4$ is the sum of the lengths of all edges $uv$ for
which $S_{uv}$ displays it.  The tree also defines a pseudometric
$d_{\mathcal T}$ on $X$: for distinct taxa $x_1,x_2$,
$d_{\mathcal T}(x_1,x_2)$ is the length of
$x_1x_1\mid x_2x_2$, and $d_{\mathcal T}(x,x):=0$.

For two edge-weighted $X$-trees
$\mathcal T_i=(T_i,\phi_i,l_i)$, define their \emph{distance covariance}
$\dCov^2(\mathcal T_1,\mathcal T_2)$ by multiplying every ordered
quadruple contribution to $Q(\mathcal T_1,\mathcal T_2)$ by the lengths
of its two induced partial quartets.  More explicitly,

\[
\dCov^2(\mathcal T_1,\mathcal T_2)
:=\sum_{\substack{tu\in E_1\\vw\in E_2}}
Q(S_{tu},S_{vw})l_1(tu)l_2(vw).
\]
This quantity is nonnegative, and it is zero exactly when the underlying
$X$-trees are independent.  Define their \emph{distance correlation} by
\[
\dCor(\mathcal T_1,\mathcal T_2):=
\frac{\dCov^2(\mathcal T_1,\mathcal T_2)}
{\sqrt{\dCov^2(\mathcal T_1,\mathcal T_1)
\dCov^2(\mathcal T_2,\mathcal T_2)}}.
\]
When the denominator is zero, set
$\dCor(\mathcal T_1,\mathcal T_2):=0$.
This ratio is the quantity often denoted by squared distance correlation;
throughout this paper we use the shorter symbol $\dCor$ for that convention.
These names are standard in statistics for measures of dependence
between random vectors; we now justify the same notation for
edge-weighted $X$-trees.

Sz\'ekely et al.~\cite{SzekelyRizzoBakirov2007} introduced distance
covariance in 2007 to detect nonlinear dependence between Euclidean
random variables.  For an $n$-point sample
$\{(x_i,y_i):1\le i\le n\}$, empirical distance covariance is defined
from the pairwise distances in the two coordinates.  Thus it can be
viewed as an invariant of two finite pseudometrics on the same ground
set.  For pseudometrics $d_1,d_2:X\times X\to\mathbb R_{\ge0}$ on
$X=\{x_1,\ldots,x_n\}$, define, for $i\in\{1,2\}$ and
$j,k\in\{1,\ldots,n\}$,
\[
D_i(j,k):=d_i(x_j,x_k)
-\frac1n\sum_{l=1}^n\bigl(d_i(x_j,x_l)+d_i(x_k,x_l)\bigr)
+\frac1{n^2}\sum_{l=1}^n\sum_{m=1}^n d_i(x_l,x_m),
\]
where the summation indices $l,m$ are taxon indices and are unrelated to the
edge-length functions $l_i$.
Then define
\[
\dCov^2(d_1,d_2):=\frac1{n^2}\sum_{j=1}^n\sum_{k=1}^n
D_1(j,k)D_2(j,k).
\]
Here $D_i(j,k)$ is a centred distance-matrix entry and is unrelated to the
directed-edge set $D_i$ introduced in Section~\ref{sec:tree-framework}.  We use the same symbol
$\dCov^2$ for the tree quantity and the pseudometric quantity because their
argument types distinguish them; Theorem~\ref{thm:tree-distance-covariance}
compares these two quantities explicitly.

\begin{theorem}\label{thm:tree-distance-covariance}
For $X=\{x_1,\ldots,x_n\}$ and any two edge-weighted $X$-trees
$\mathcal T_1,\mathcal T_2$,
\[
\dCov^2(\mathcal T_1,\mathcal T_2)
=\frac32\dCov^2(d_{\mathcal T_1},d_{\mathcal T_2}).
\]
\end{theorem}

\begin{proof}
A split $S=A\mid B$ defines the \emph{split metric} $d_S$, where
$d_S(x_1,x_2)=0$ if $x_1,x_2$ lie on the same side and $d_S(x_1,x_2)=1$
otherwise.  For an edge-weighted tree,
\[
d_{\mathcal T_i}=\sum_{uv\in E(T_i)}l_i(uv)d_{S_{uv}}.
\]
Because metric distance covariance is bilinear in the two distance
matrices, it suffices to consider two one-edge trees whose edges have
length $1$.  Let their splits be $R_1\mid R_2$ and $C_1\mid C_2$, put
$A_i^j:=R_i\cap C_j$, and let $r_i,c_j,a_i^j$ denote the corresponding
cardinalities divided by $n$.  Write $A=(a_i^j)_{i,j\in\{1,2\}}$ and
$\Delta:=\det A$.
Then
\[
\dCov^2(\mathcal T_1,\mathcal T_2)
=Q(R_1\mid R_2,C_1\mid C_2)=6\Delta^2,
\]
so it remains to prove
\[
\dCov^2(d_{R_1\mid R_2},d_{C_1\mid C_2})=4\Delta^2.
\]
For brevity, let $d_1=d_{R_1\mid R_2}$ and
$d_2=d_{C_1\mid C_2}$.  Given $i,j\in\{1,2\}$, write $\bar i,\bar j$
for their complementary indices.  If $x_k,x_l\in R_i$, then
\[
D_1(k,l)=\frac1{n^2}
\bigl(-2n|R_{\bar i}|+2|R_1||R_2|\bigr)=-2r_{\bar i}^{\,2},
\]
whereas if $x_k$ and $x_l$ are separated by $R_1\mid R_2$, then

\[
D_1(k,l)=\frac1{n^2}
\bigl(n^2-n(|R_1|+|R_2|)+2|R_1||R_2|\bigr)=2r_1r_2.
\]
The corresponding formulas for $D_2$ are obtained by replacing $R,r$
with $C,c$.  Consequently,
\begin{align*}
\dCov^2(d_1,d_2)=4\Biggl(&
\sum_{i,j\in\{1,2\}}(a_i^j r_{\bar i}c_{\bar j})^2
-2\sum_{i=1}^2a_i^1a_i^2r_{\bar i}^{\,2}c_1c_2\\
&-2\sum_{j=1}^2a_1^ja_2^jr_1r_2c_{\bar j}^{\,2}
+2(a_1^1a_2^2+a_1^2a_2^1)r_1r_2c_1c_2\Biggr).
\end{align*}
Since the entries of $A$ sum to $1$,
\[
a_i^j=r_ic_j+(-1)^{i+j}\Delta.
\]
Substituting this relation for every entry, all terms except the
multiples of $\Delta^2$ cancel, and one obtains
$\dCov^2(d_1,d_2)=4\Delta^2$.
\end{proof}

The main result of \cite{SzekelyRizzoBakirov2007} states that empirical
distance covariance estimates the distance covariance of the underlying
joint distribution when both vectors have finite first moments.  This
population quantity is nonnegative and vanishes exactly under
independence.  The result was extended from Euclidean spaces to metric
spaces of strong negative type by Lyons~\cite{Lyons2013}.  For finite
metrics this class agrees with strictly negative-type spaces, and tree
metrics belong to it~\cite{HjorthEtAl1998}.

A pair $(\mathcal T_1,\mathcal T_2)$ of edge-weighted $X$-trees defines
a discrete joint law by drawing $U$ uniformly from $X$ and setting
$Z_i=\phi_i(U)$ for $i\in\{1,2\}$.  The mass of
$(v_1,v_2)$ is
\[
\frac{|\phi_1^{-1}(v_1)\cap\phi_2^{-1}(v_2)|}{|X|}.
\]
The random vertices are independent if and only if the trees are independent,
by Lemma~\ref{lem:vertex-side-bridge}.  Thus the preceding results justify
using $\dCov^2$ and $\dCor$ as measures of dependence between edge-weighted
trees.

\subsection{Strict positivity of Bergsma--Dassios
\texorpdfstring{$\tau^*$}{tau-star} under dependence}
\label{subsec:tau-star}

On path trees, the unweighted quartet covariance is directly related to the
sign covariance introduced by Bergsma and
Dassios~\cite{BergsmaDassios2014} as a measure of dependence between ordinal
random variables, extending Kendall's $\tau$.  They write $t^*$ for a finite
sample and $\tau^*$ for the
expected value of $t^*$ on four independent observations.  Because this
statistic is invariant under reversing either order, the weak order
of a finite sample corresponds naturally to a labelled tree whose underlying
graph is a path.  Bergsma and Dassios proved that $\tau^*\ge0$, with
equality if and only if the variables are independent, when the joint law
is discrete, jointly absolutely continuous, or a mixture of these two
types, and conjectured the same conclusion for arbitrary bivariate laws.
Drton, Han, and Shi~\cite{DrtonHanShi2020} subsequently established both
conclusions for random vectors with continuous margins, allowing joint
laws that need not be absolutely continuous.  We now establish both
conclusions for arbitrary real-valued bivariate laws.

For the unnormalised Bergsma--Dassios convention and
$z=(z_1,z_2,z_3,z_4)$, let $r,s,t,u$ be pairwise distinct occurrence indices
with $\{r,s,t,u\}=\{1,2,3,4\}$; the notation $rs\mid tu$ divides the four
occurrences into the two indicated pairs.  Set
\[
\begin{aligned}
I_{rs\mid tu}(z)&:=\mathbf{1}\!\left\{
\max(z_r,z_s)<\min(z_t,z_u)\text{ or }
\max(z_t,z_u)<\min(z_r,z_s)\right\},\\
a(z)&:=I_{13\mid24}(z)-I_{12\mid34}(z).
\end{aligned}
\]
The indicator $I_{rs\mid tu}$ and the kernel $a(z)$ are local to this
subsection.  They are unrelated to the earlier index sets $I_{t,u}$ and
$I'_{t,u}$ and to the coefficient entries $a_{t,u}^{v,w}$.
For four independent copies
$(X^{(1)},Y^{(1)}),\ldots,(X^{(4)},Y^{(4)})$ of $(X,Y)$, define
\[
\tau^*(X,Y):=\mathbb E\!\left[
a(X^{(1)},X^{(2)},X^{(3)},X^{(4)})
a(Y^{(1)},Y^{(2)},Y^{(3)},Y^{(4)})\right].
\]

\begin{corollary}\label{cor:tau-star}
Let $(X,Y)$ be a real-valued random vector.  If $X$ and $Y$ are not
independent, then $\tau^*(X,Y)>0$.
\end{corollary}

\begin{proof}
For the distribution-function discrepancy, write
$D(x,y):=F(x,y)-F_X(x)\allowbreak F_Y(y)$, where
$F(x,y)=P(X\le x,Y\le y)$, $F_X(x)=P(X\le x)$, and
$F_Y(y)=P(Y\le y)$.  This local symbol $D$ is
unrelated to the directed-edge sets $D_i$ and to the centred distance entries
$D_i(j,k)$ used above.  Introduce the unnormalised
Blum--Kiefer--Rosenblatt discrepancy
\begin{equation}\label{eq:BKR-discrepancy}
\mathcal B(X,Y):=\int_{\mathbb R^2}D(x,y)^2\,dF_X(x)\,dF_Y(y).
\end{equation}
We prove the quantitative inequality
\begin{equation}\label{eq:tau-star-BKR-lower-bound}
\tau^*(X,Y)\ge2\mathcal B(X,Y).
\end{equation}

Suppose first that $X$ and $Y$ have finite ordered marginal supports
$x_1<\cdots<x_m$ and $y_1<\cdots<y_n$.
We call each position $(x_i,y_j)$ in the resulting marginal-support grid a cell,
including positions having zero joint mass.
Put $p_i=P(X=x_i)>0$, $q_j=P(Y=y_j)>0$, and, for $i<m$, $j<n$,
\[
\Delta_{ij}:=P(X\le x_i,Y\le y_j)-P(X\le x_i)P(Y\le y_j).
\]
The symbols $p_i$ and $q_j$ are local marginal masses in this subsection and
are unrelated to the coefficient families $p$ and $q$ from
Section~\ref{sec:tree-framework}.
Assume initially that all cell probabilities are rational.  Replace
each cell mass by the corresponding number of equally weighted labels, using
one common finite label set for the two coordinates.  Let $T_X$ denote the
labelled path tree whose $i$th vertex carries exactly the labels in row $i$,
and let $T_Y$ denote the labelled path tree whose $j$th vertex carries exactly
the labels in column $j$.  Thus they are labelled trees of the type defined in
Section~\ref{sec:tree-framework}, with the replicated label set playing the
role of the taxon set.

For a realised coordinate vector $z=(z_1,z_2,z_3,z_4)$, write
$z_{(1)}\le z_{(2)}\le z_{(3)}\le z_{(4)}$ for its order statistics.  When
$z_{(2)}<z_{(3)}$, define its \emph{strict middle-cut partition}
\[
\Pi_z:=
\bigl\{\{r:z_r\le z_{(2)}\},\ \{r:z_r\ge z_{(3)}\}\bigr\};
\]
both blocks have two occurrence indices.  In the corresponding coordinate
path tree, an edge separates the four occurrence indices into two pairs
exactly when this strict inequality holds; in that case the induced positional
partition is the unique partition $\Pi_z$.  Repeated
occurrences of one label have the same $X$- and $Y$-coordinates and therefore
lie in the same block of every strict middle-cut partition.  Equal $X$- and
$Y$-partitions give one common partial quartet.  If the partitions differ,
their common refinement has four singleton blocks, so no label can be repeated
and the displayed objects are two different full quartets.  If either
coordinate has no strict middle cut, the ordered quadruple contributes zero
both to the quartet-covariance count and to the product kernel.  This identifies
the relevant events for all multiplicities and ties.

Let $\Pi_X$ and $\Pi_Y$ denote the strict middle-cut partitions of the four
$X$- and $Y$-coordinates when they are defined, and set
\begin{align*}
C_4&:=\{\Pi_X\text{ and }\Pi_Y\text{ are defined and }\Pi_X=\Pi_Y\},\\
D_4&:=\{\Pi_X\text{ and }\Pi_Y\text{ are defined and }\Pi_X\ne\Pi_Y\}.
\end{align*}

To determine the coefficients of $P(C_4)$ and $P(D_4)$ in the representation
of $\tau^*$, symmetrise over the four occurrence indices.  Conditional on a common
strict partition, a uniform permutation of the four occurrence indices
sends it to $12\mid34$, $13\mid24$, or $14\mid23$ equally often.  The
corresponding values of $a$ are $-1,+1,0$, so the conditional mean of the
product is $2/3$.  Conditional on two distinct strict partitions, the
six ordered pairs of distinct partitions occur equally often; two
products are $-1$ and four are zero, so the conditional mean is $-1/3$.
Exchangeability therefore gives
\[
\tau^*(X,Y)=\frac23P(C_4)-\frac13P(D_4).
\]
By the definition of quartet covariance,
\begin{equation}\label{eq:path-Q-tau-star}
Q(T_X,T_Y)=P(C_4)-\frac12P(D_4)=\frac32\tau^*(X,Y).
\end{equation}

For $1\le i<m$, let $S_i^X$ be the split displayed by the edge of $T_X$
between the vertices carrying $x_i$ and $x_{i+1}$.  For $1\le j<n$, define
$S_j^Y$ analogously from $T_Y$.  Here an endpoint mass means the normalized
label mass $r_{t,u}(u)$ or $c^{v,w}(w)$ of the head vertex's label fibre in a
directed-edge orientation.  For the first edge, the orientation whose head is
the vertex carrying $x_i$ contributes $p_i$, and the reverse orientation
contributes $p_{i+1}$.  For the second edge the analogous two contributions
are $q_j$ and $q_{j+1}$.  Thus the two orientation sums are
$p_i+p_{i+1}$ and $q_j+q_{j+1}$, respectively.  The corresponding quadrant
determinant is
\[
\det\begin{pmatrix}
P(X\le x_i,Y\le y_j)&P(X\le x_i,Y>y_j)\\
P(X>x_i,Y\le y_j)&P(X>x_i,Y>y_j)
\end{pmatrix}=\Delta_{ij}.
\]
Thus $Q(S_i^X,S_j^Y)=6\Delta_{ij}^2$, and
Lemma~\ref{lem:coefficient-lower-bound} gives
\[
Q(T_X,T_Y)\ge
3\sum_{i=1}^{m-1}\sum_{j=1}^{n-1}
(p_i+p_{i+1})(q_j+q_{j+1})\Delta_{ij}^2.
\]
Together with \eqref{eq:path-Q-tau-star},
\begin{align}
\tau^*(X,Y)
&\ge2\sum_{i=1}^{m-1}\sum_{j=1}^{n-1}
(p_i+p_{i+1})(q_j+q_{j+1})\Delta_{ij}^2\notag\\
&\ge2\sum_{i=1}^{m-1}\sum_{j=1}^{n-1}p_iq_j\Delta_{ij}^2
=2\mathcal B(X,Y).                                      \label{eq:finite-tau-BKR}
\end{align}
The atomic formula for $\mathcal B$ also has terms with $i=m$ or $j=n$,
but $D(x_m,y_j)=D(x_i,y_n)=0$, so those boundary terms vanish.
Moreover, if the finite-support law is dependent, then some
$\Delta_{ij}$ is nonzero.  Since every displayed marginal mass is positive,
the second inequality in \eqref{eq:finite-tau-BKR} is then strict.  Thus, for
finite-support laws, equality $\tau^*=2\mathcal B$ occurs only under
independence, when both sides are zero.

The finite-support result extends to arbitrary, not necessarily rational,
cell probabilities.  Keep
zero cells at zero and let
$\mathcal S:=\{(i,j):P(X=x_i,Y=y_j)>0\}$.
For a cell $s=(i,j)\in\mathcal S$, abbreviate its probability by
$p_s:=P(X=x_i,Y=y_j)$.  Choose $s_0\in\mathcal S$ and, for every sufficiently
large positive integer $N$, put
\[
p_s^{(N)}=\frac{\lfloor Np_s\rfloor}{N}
\quad(s\in\mathcal S\setminus\{s_0\}),
\qquad
p_{s_0}^{(N)}=1-\sum_{s\in\mathcal S\setminus\{s_0\}}p_s^{(N)}.
\]
Every originally positive cell, and hence every occupied marginal row
and column, remains positive for sufficiently large $N$.  These rational
tables converge to the original table.  Both sides of
\eqref{eq:finite-tau-BKR} are continuous polynomial functions of the
fixed table entries, so the inequality passes to the limit.

We now remove the finite-support assumption.  For each integer $k\ge1$, define
the nested finite partition
\[
\mathcal P_k:=\{(-\infty,-k]\}
\cup\{(j2^{-k},(j+1)2^{-k}]:j\in\mathbb Z,\ -k2^k\le j<k2^k\}
\cup\{(k,\infty)\}.
\]
Every endpoint of $\mathcal P_k$ is an endpoint of
$\mathcal P_{k+1}$.  Form the path tree for the $X$-coordinate from the cells
of $\mathcal P_k$
having positive $F_X$-mass, in their inherited order, and define $g_k(x)$
to be the ordinal rank of the occupied cell containing $x$; define it
arbitrarily on the union of zero-$F_X$-mass cells.  Define $h_k$ analogously
from the occupied $F_Y$-cells, and set $X_k=g_k(X)$, $Y_k=h_k(Y)$.  Almost
surely, equal entries remain equal under every quantizer, while any distinct
pair is eventually separated.  The weak-order pattern of each sampled
four-tuple---the collection of all pairwise relations $<$, $=$, and $>$ among
its four coordinates---therefore eventually stabilises.
The $\tau^*$ kernel is bounded, so dominated convergence gives
\begin{equation}\label{eq:tau-quantisation-limit}
\tau^*(X_k,Y_k)\longrightarrow\tau^*(X,Y).
\end{equation}

Let $r_k(x),s_k(y)$ denote the upper endpoints of the cells containing
$x,y$, with a top tail interpreted as $+\infty$.  These local endpoint maps
are unrelated to the coefficient families $r$ and $s$ in
Section~\ref{sec:tree-framework}.  Extend
$D$ by
\[
\bar D(+\infty,y)=\bar D(x,+\infty)=\bar D(+\infty,+\infty)=0,
\qquad
\bar D(x,y)=D(x,y)\quad(x,y\in\mathbb R).
\]
Because the cells are right closed, for $F_X\times F_Y$-almost every
$(x,y)$,
\[
F_{X_k,Y_k}(g_k(x),h_k(y))-F_{X_k}(g_k(x))F_{Y_k}(h_k(y))
=\bar D(r_k(x),s_k(y)).
\]
If $X'$ and $Y'$ are independent with laws $F_X$ and $F_Y$, then
\[
\mathcal B(X_k,Y_k)
=\mathbb E\!\left[\bar D(r_k(X'),s_k(Y'))^2\right].
\]
For finite $x,y$, eventually neither lies in a tail cell.  Refinement
makes the upper endpoints nonincreasing, and
$0\le r_k(x)-x<2^{-k}$ and $0\le s_k(y)-y<2^{-k}$,
so $r_k(x)\downarrow x$ and $s_k(y)\downarrow y$.  The rectangles
$(-\infty,r_k(x)]\times(-\infty,s_k(y)]$ decrease to
$(-\infty,x]\times(-\infty,y]$.  Continuity from above of the joint and
marginal measures, followed by dominated convergence, implies
\begin{equation}\label{eq:BKR-quantisation-limit}
\mathcal B(X_k,Y_k)\longrightarrow\mathcal B(X,Y).
\end{equation}
Applying \eqref{eq:finite-tau-BKR} to $(X_k,Y_k)$ and using
\eqref{eq:tau-quantisation-limit} and
\eqref{eq:BKR-quantisation-limit} proves
\eqref{eq:tau-star-BKR-lower-bound}.

It remains to verify that $\mathcal B$ detects every failure of
independence, including atomic and singular laws.  For all
$x,x',y,y'\in\mathbb R$,
\begin{equation}\label{eq:D-variation}
|D(x,y)-D(x',y')|
\le2|F_X(x)-F_X(x')|+2|F_Y(y)-F_Y(y')|.
\end{equation}
Suppose $D(x_0,y_0)=d\ne0$.  Necessarily
$F_X(x_0),F_Y(y_0)\in(0,1)$.  Choose
$0<\varepsilon<\min\{|d|/8,\allowbreak F_X(x_0),\allowbreak F_Y(y_0)\}$.
The Borel sets
\[
A:=\{x\le x_0:F_X(x_0)-\varepsilon<F_X(x)\le F_X(x_0)\},
\]
\[
C:=\{y\le y_0:F_Y(y_0)-\varepsilon<F_Y(y)\le F_Y(y_0)\}
\]
have positive $F_X$- and $F_Y$-measure.  Indeed, if
$U=F_X(X)$ and $t=F_X(x_0)-\varepsilon$, then $U$ need not be uniform when
$F_X$ has atoms, but it satisfies the super-uniform inequality
$P(U\le t)\le t$.  Hence
$P(X\in A)\ge\varepsilon$, and similarly $P(Y\in C)\ge\varepsilon$.
By \eqref{eq:D-variation}, $|D(x,y)|>|d|/2$ on $A\times C$, so
$\mathcal B(X,Y)\ge\frac{d^2}{4}P(X\in A)P(Y\in C)>0$.
Thus $\mathcal B=0$ forces $D(x,y)=0$ for all $x,y$.  Equality of the
joint and product measures on all lower-left rectangles implies equality
on the Borel sigma-field.  Conversely independence gives $D\equiv0$.
Therefore $\mathcal B=0$ if and only if $X$ and $Y$ are independent.
Dependence gives $\mathcal B>0$, and
\eqref{eq:tau-star-BKR-lower-bound} gives $\tau^*(X,Y)>0$.
If $X$ and $Y$ are independent, then the four-coordinate
$X$-array and $Y$-array in the definition of $\tau^*$ are independent.
Exchangeability gives zero expectation to each factor $a$, so
$\tau^*(X,Y)=0$.  Thus Corollary~\ref{cor:tau-star} is equivalent to the
zero characterisation stated in the title and introduction.
\end{proof}

\section{Discussion}

Corollary~\ref{cor:tau-star} is a population identification result.  It
extends the zero characterisation of $\tau^*$ to arbitrary real bivariate
laws.  No new sample-level limit theorem is proved here, and consequences for
test consistency under atoms or singular components are not analysed.  The
four-sample kernel and the distribution-function
discrepancy are bounded.  Together with right-closed nested quantisation,
continuity from above, and dominated convergence, this avoids moment,
density, and continuity assumptions and treats atomic and singular laws
directly.

The phylogenetic and statistical results use the same construction.  Quartet
covariance is first written as a sum of local nonnegative squares on general
labelled trees.  Path trees then encode the weak-order patterns of finite
bivariate distributions, and the coefficient lower bound converts those
squares into a distribution-function discrepancy.  This finite inequality is
the input to the approximation argument for arbitrary real bivariate laws.

\begin{acks}[Acknowledgments]
The majority of this work was carried out while both authors were affiliated
with the CAS-MPG Partner Institute for Computational Biology (PICB), Shanghai
Institutes for Biological Sciences, Chinese Academy of Sciences, Shanghai,
China.

The authors used OpenAI's GPT-5.6 Sol with the ultra reasoning-effort
setting to assist in completing the proofs of
Lemma~\ref{lem:coefficient-lower-bound} and Corollary~\ref{cor:tau-star}, and in
constructing the accompanying formalisation in Lean~4 and mathlib.  The authors
checked and revised all mathematical arguments and formal code and take full
responsibility for the contents of the manuscript and the formalisation.
\end{acks}

\begin{supplement}
\stitle{Lean 4 formalisation and reproducibility source code}
\sdescription{The accompanying supplementary ZIP file contains the complete
Lean 4 source closure for the machine-checked results cited in the manuscript,
together with pinned Lean and mathlib versions, a manuscript-to-Lean
correspondence table, an axiom audit, integrity hashes, and licence
information.}
\end{supplement}

\end{document}